\documentclass[smallextended,envcountsame]{svjour3}%

\usepackage{amsfonts}
\usepackage{geometry}
\usepackage{amsmath}
\usepackage{amssymb}
\usepackage{graphicx}
\smartqed

\begin{document}

\title{Discrete Cycloids from Convex Symmetric Polygons}

\author{Marcos Craizer \and Ralph Teixeira \and Vitor Balestro}

\institute{M. Craizer
              \at PUC-Rio, Brazil \\\email{craizer@mat.puc-rio.br}
           \and           
           R. Teixeira
              \at UFF, Brazil \\\email{ralph@mat.uff.br}
           \and
           V. Balestro
              \at UFF and CEFET, Brazil \\\email{vitorbalestro@id.uff.br}
}

\maketitle

\begin{abstract}
Cycloids, hipocycloids and epicycloids have an often forgotten common
property: they are homothetic to their evolutes. But what if use convex
symmetric polygons as unit balls, can we define evolutes and cycloids which
are genuinely discrete? Indeed, we can! We define \emph{discrete cycloids} as
eigenvectors of a discrete \emph{double evolute transform} which can be seen
as a linear operator on a vector space we call \emph{curvature radius space}.
We are also able to classify such cycloids according to the eigenvalues of
that transform, and show that the number of \emph{cusps} of each cycloid is
well determined by the ordering of those eigenvalues. As an elegant
application, we easily establish a version of the four-vertex theorem for
closed convex polygons. The whole theory is developed using only linear
algebra, and concrete examples are given.



\end{abstract}

\section{Introduction}

Euclidean cycloids (and hypocycloids and epicycloids) can be characterized as
planar curves which are $\lambda-$homothetic to their evolutes (which kind of
cycloid depends on the signal of $\lambda-1$). Such idea can na\-tu\-rally be
extended to normed planes (see \cite{Cycloids}) if the unit ball is
sufficiently smooth. But what if the ball is a polygon? Can we define
genuinely discrete analogues of cycloids without using limiting processes?

In this paper, we present a discrete evolute transform  and then define discrete cycloids as polygonal lines which
are homothetic to their double evolutes.  Our evolute construction is an alternative to \cite{Arnold} -- we follow instead the
approach in \cite{Evolutes}. By using a suitable representation,
we can represent our polygonal lines as vectors in a space we call $L_{P}$. In
this space, the double evolute transform is linear, so we rephrase the problem
in two different ways: as an eigenvector problem, or as a recurrence. By
analysing the interaction of this transform within many subspaces of $L_{P}$,
we are able to establish the spectral representation of the double evolute
transform. As a consequence, we are able to decompose any element of $L_{P}$
as a sum of cycloids, producing a generalization of the Discrete Fourier
Transform. As an application of this decomposition, we provide an elegant
proof of a four vertex theorem for polygons.

We provide explicit cycloid formulae when the polygon is regular; as expected,
if the polygon is close to the usual euclidean unit ball, the corresponding
cycloids are approximations of the classical cycloids.

More specifically, we start with a symmetric polygon $P=P_{1}P_{2}...P_{2n}$
which will be our unit ball on the plane, and define its dual ball
$Q=Q_{1}Q_{2}...Q_{2n}$. A polygonal line $M$ whose sides are respectively
parallel to the sides of $P$ can then be represented by the lengths
$r_{1},r_{2},...$ of its sides (taking the sides of $P$ as unit length in each
direction) -- we call this the \emph{curvature radius representation} of $M$.
We show that $M$ is a cycloid exactly when its radii satisfy a difference
equation of the kind%
\begin{equation}
\frac{-1}{\left[  P_{i},P_{i+1}\right]  }\nabla_{i}\left(  \frac{\Delta_{i}%
r}{\left[  Q_{i},Q_{i+1}\right]  }\right)  =\lambda r_{i} \label{DiffEq}%
\end{equation}
where $\Delta_{i}$ and $\nabla_{i}$ are forward and backward difference
operators, and $\left[  \cdot,\cdot\right]  $ is the determinant of two
vectors. This equation is the natural discretization of the 2nd order
Sturm-Liouville type Differential Equation displayed in \cite{Cycloids}, so we
can reasonably expect it to have a nice spectral structure (as seen in
\cite{WangShi}). Indeed, if we require the list $r_{i}$ to be $2n$-periodic,
we are able to show that the eigenvalues associated to Eq. (\ref{DiffEq}) can
be ordered as%
\[
\lambda_{0}(=0)<\lambda_{1}^{1}=\lambda_{1}^{2}(=1)<\lambda_{2}^{1}\leq
\lambda_{2}^{2}<\lambda_{3}^{1}\leq\lambda_{3}^{2}<...\lambda_{n-1}^{1}%
\leq\lambda_{n-1}^{2}<\lambda_{n}%
\]
where each eigenvector is a closed cycloid -- except for those with eigenvalue
$1$, which form a $2$-dimensional space of non-closed cycloids (there are no
non-trivial hypocycloids). Moreover, we show that the cycloid associated to
$\lambda_{k}^{i}$ has exactly $2k$ ordinary cusps. When $k$ is even, the
cycloid is a symmetric polygon; when $k$ is odd ($\neq1$), the cycloid is a
polygon of $0$-width. It is interesting to note that, in the Euclidean case,
the eigenvalue $\lambda$ is \textbf{determined} by the number of cusps of the
cycloids -- not the case here, since we might have $\lambda_{k}^{1}\neq
\lambda_{k}^{2}$!

Figure \ref{FigExample} shows a concrete example, where the unit ball is the
octagon in the top image (and the cycloid for $\lambda_{0}=0$). The next row
pictures two different open cycloids ($\lambda=1$). The third row shows
symmetric cycloids ($\lambda_{2}^{1}\approx3.21$ and $\lambda_{2}^{2}%
\approx3.46$), each with $4$ cusps. The fourth row shows $0$-width cycloids
($\lambda_{3}^{1}\approx4.90$ and $\lambda_{3}^{2}\approx7.92$), with $3$
"double" cusps each. Finally, we have one last cycloid ($\lambda_{4}%
\approx8.15$) where all $8$ vertices are cusps.

\begin{center}
\begin{figure}[h] \centering
{\includegraphics[height=1.2 in,width=2.4 in]{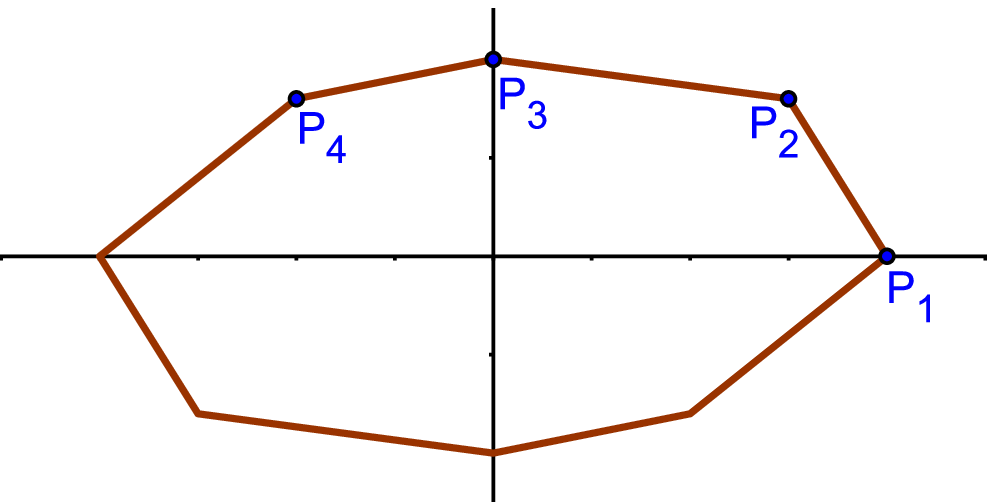}}

{\includegraphics[height=1.2 in,width=2.4 in]{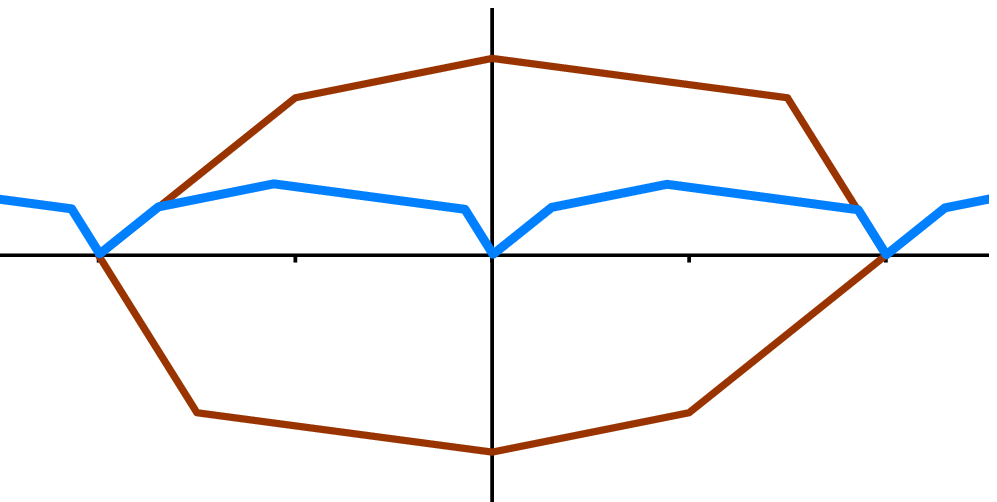}}
{\includegraphics[height=1.2 in,width=2.4 in]{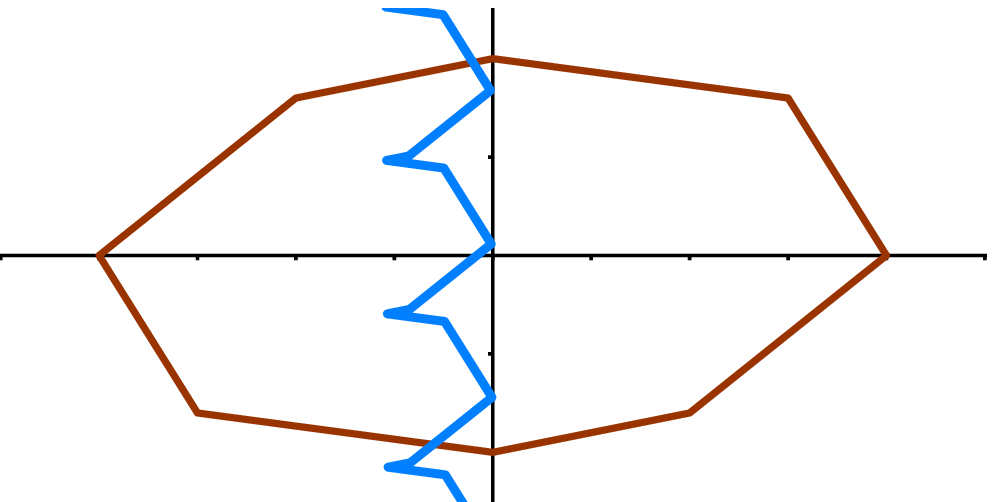}}

{\includegraphics[height=1.2 in,width=2.4 in]{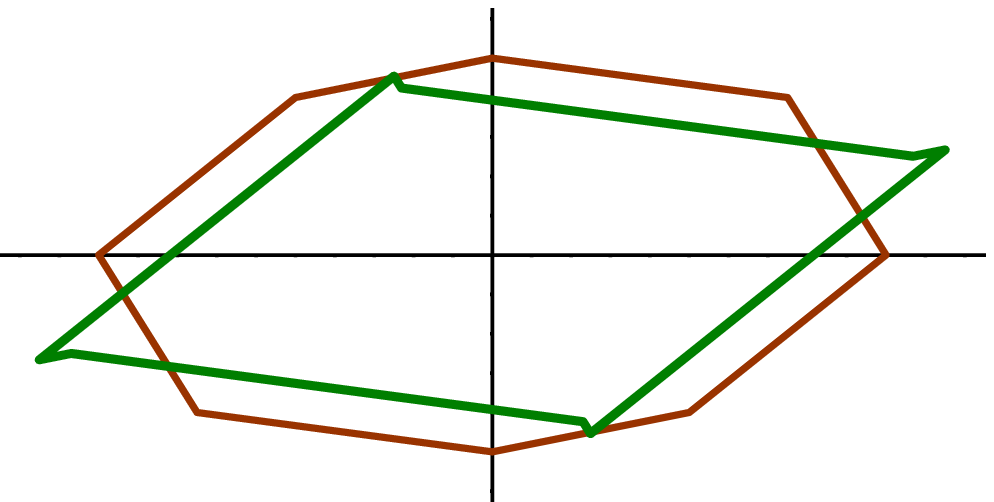}}
{\includegraphics[height=1.2 in,width=2.4 in]{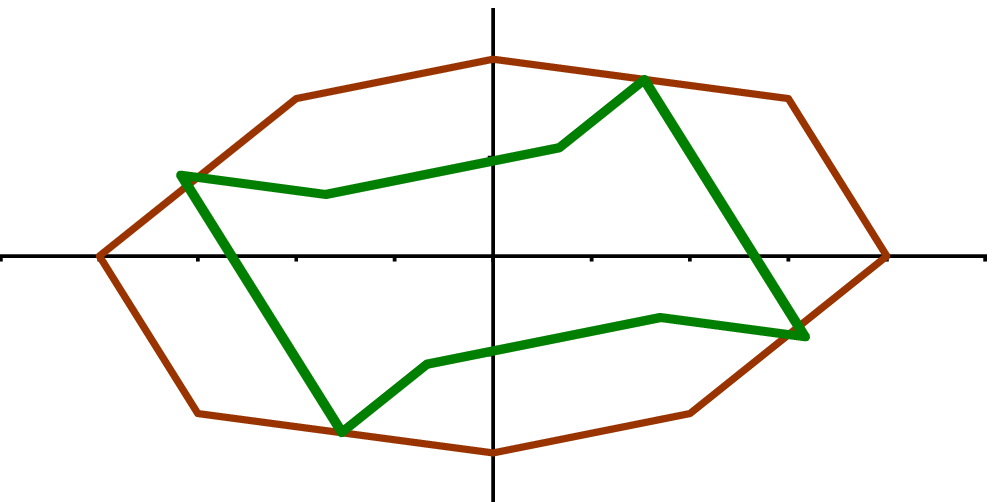}}

{\includegraphics[height=1.2 in,width=2.4 in]{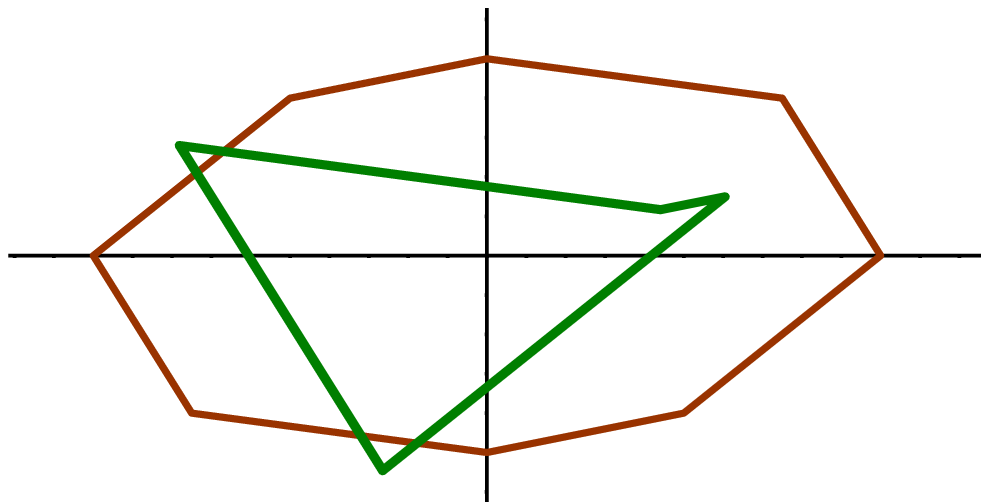}}
{\includegraphics[height=1.2 in,width=2.4 in]{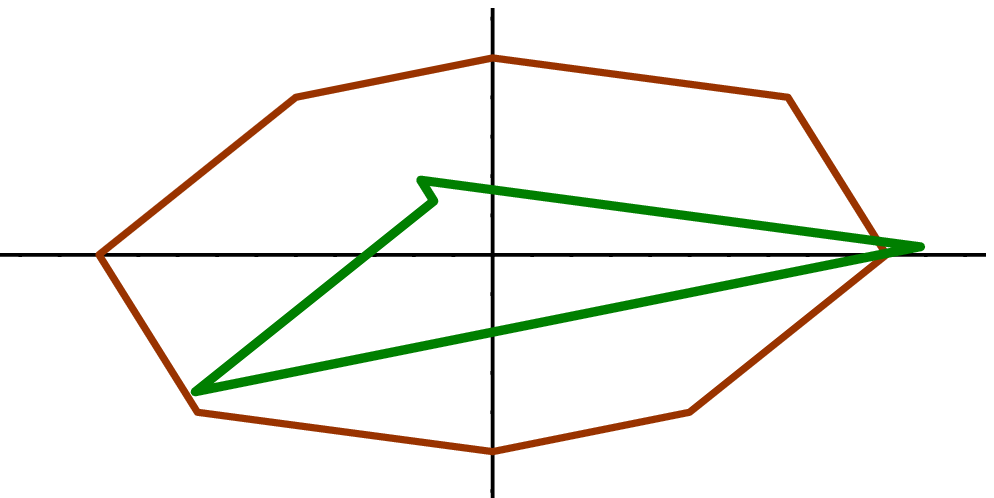}}

{\includegraphics[height=1.2 in,width=2.4 in]{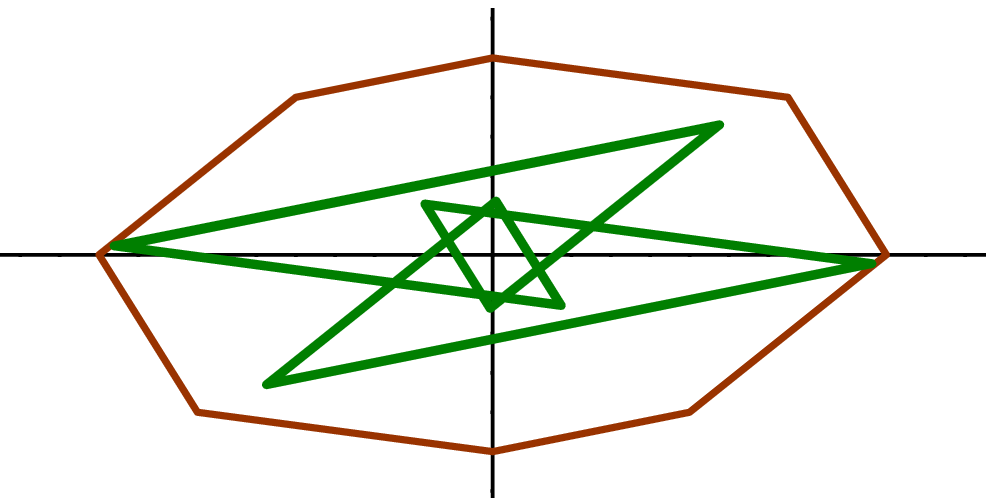}}
\caption{Cycloids related to an octagon}
\label{FigExample}
\end{figure}
\end{center}

If we require instead the list $r_{i}$ to be just $2mn$-periodic, the old
eigenvalues are joined by new ones in a very orderly fashion%
\[%
\setlength{\tabcolsep}{0pt}
\begin{tabular}
{llllll}
&  &  &  & $\lambda_{0}\left(  =0\right)  $ & $<$\\
$<\lambda_{1/m}^{1}=\lambda_{1/m}^{2}$ & $<\lambda_{2/m}^{1}=\lambda_{2/m}%
^{2}$ & $<...$ & $<\lambda_{1-1/m}^{1}=\lambda_{1-1/m}^{2}<$ & $\lambda
_{1}^{1}=\lambda_{1}^{2}\left(  =1\right)  $ & $<$\\
$<\lambda_{1+1/m}^{1}=\lambda_{1+1/m}^{2}$ & $<\lambda_{1+2/m}^{1}%
=\lambda_{1+2/m}^{2}$ & $<...$ & $<\lambda_{2-1/m}^{1}=\lambda_{2-1/m}^{2}<$ &
$\lambda_{2}^{1}\leq\lambda_{2}^{2}$ & $<$\\
$<\lambda_{2+1/m}^{1}=\lambda_{2+1/m}^{2}$ & $<\lambda_{2+2/m}^{1}%
=\lambda_{2+2/m}^{2}$ & $<...$ & $<\lambda_{3-1/m}^{1}=\lambda_{3-1/m}^{2}<$ &
$\lambda_{3}^{1}\leq\lambda_{3}^{2}$ & $<$\\
$<...$ & $...$ & $...$ & $...$ & $...$ & $<$\\
$<\lambda_{n-1+1/m}^{1}=\lambda_{n-1+1/m}^{2}$ & $<\lambda_{n-1+2/m}%
^{1}=\lambda_{n-1+2/m}^{2}$ & $<...$ & $<\lambda_{n-1/m}^{1}=\lambda
_{n-1/m}^{2}< $ & $\lambda_{n}$ &
\end{tabular}
\
\]
(note the previous eigenvalue list in the last column -- they \textbf{do not}
necessarily come in identical pairs). In summary, we have $2m-2$ non-trivial
hypocycloids, $2$ cycloids and $2mn-2m-1$ epicycloids. The fractional indices
work across different periods -- for example, the eigenvalue $\lambda
_{2/5}^{1}$ which appears when looking for a list of period $10n$ is indeed
the same as $\lambda_{6/15}^{1}$ which would show up when looking for period
$30n$.

Many of the results we present have been previously established in a more
general context -- see \cite{WangShi}, for example -- but our approach is very
geometric and requires only basic Linear Algebra.

The organization of the paper is as follows: in section \ref{SecMink} we
establish basic facts about \emph{the polygonal unit ball and its dual}.
Section \ref{SecCurv} presents the \emph{curvature radius space} (with an
inner product) and establishes geometric interpretations of many of its
interesting subspaces. In section \ref{SecEvol} we present the discrete
\emph{evolute }and the\emph{\ double evolute }transforms. Section
\ref{SecCycl} finally defines the \emph{cycloids}, starting out with the
important example when the ball is a regular polygon, and proceeds to the
spectral analysis of the double evolute transform and the determination of the
number of cusps of each cycloid (in the case where the cycloids have the same
period as the unit ball). In section \ref{SecPeri}, we show that the general
case where our cycloids have \emph{other periods} than the unit ball can be
reduced to the $2n$-periodic polygons, establishing the more general
result. Finally, in section \ref{SecFourVertex} we use our framework to
quickly prove a four vertex theorem for polygons.

\subsection{Notation}

Given two points or vectors $v,w\in\mathbb{R}^{2}$, we denote the
\emph{determinant} whose columns are $v$ and $w$ by $\left[  v,w\right]  $.
Given a list of numbers (or vectors) $\left\{  L_{i}\right\}  _{i\in
\mathbb{Z}}$, we define the \emph{forward difference operator} $\Delta
L=\left\{  \Delta_{i}L\right\}  _{i\in\mathbb{Z}}$ where $\Delta_{i}%
L=L_{i+1}-L_{i}$ and the \emph{backward difference operator} $\nabla
L=\left\{  \nabla_{i}L\right\}  _{i\in\mathbb{Z}}$ where $\nabla_{i}%
=\Delta_{i-1}$ (whenever necessary, indices are taken mod $2n$); their
composition will be denoted $\delta^{2}$ (so $\delta^{2}L=L_{i+1}%
-2L_{i}+L_{i-1}$). Whenever two matrices $A$ and $B$ are similar (that is,
$A=M^{-1}BM$ for some invertible matrix $M$), we write $A\approx B$.

\section{The polygonal ball and its dual\label{SecMink}}

We start out with a plane symmetric non-degenerate convex polygon
$P=P_{1}P_{2}...P_{2n}$ (that is, $P_{i+n}=-P_{i}$), which will be our
reference \emph{ball of radius 1} -- we might as well imagine that the
vertices are in counter-clockwise order around the origin\footnote{Choosing a
convex symmetric unit ball is the same as choosing a norm in the plane, so we
are in the realms of Minkowski Geometry \cite{Thompson}.}. Its \emph{dual} is
the only polygon $Q=Q_{1}Q_{2}...Q_{2n}$ which satisfies%
\[
\left[  P_{i},Q_{i}\right]  =1\text{ and }\left[  \Delta_{i}P,Q_{i}\right]
=\left[  P_{i+1},\Delta_{i}Q\right]  =0
\]
for $i=1,...,2n$. Note that, in view of the first equation, the two last ones
are actually equivalent, since%
\[
\Delta_{i}\left[  P_{i},Q_{i}\right]  =\left[  \Delta_{i}P,Q_{i}\right]
+\left[  P_{i+1},\Delta_{i}Q\right]  =0\text{.}%
\]

\begin{figure}[h]
\centering
\includegraphics[height=1.5 in,width=3 in]{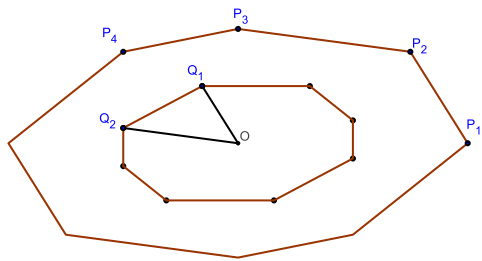}
\caption{A $P$-ball and its dual}
\label{FigBallandDual}
\end{figure}

Geometrically, each vector $Q_{i}$ is parallel to the corresponding side of
$P$, and vice-versa. These conditions actually allow us to write explicitly%
\begin{equation}
P_{i+1}=-\alpha_{i}\Delta_{i}Q\text{ and }Q_{i}=\beta_{i}\Delta_{i}P
\label{PQ}%
\end{equation}
where%
\begin{equation}
\alpha_{i}=\frac{1}{\left[  Q_{i},Q_{i+1}\right]  }\text{ and }\beta_{i}%
=\frac{1}{\left[  P_{i},P_{i+1}\right]  } \label{AlphaandBeta}%
\end{equation}
From there, it is easy to see that the dual ball is also symmetric. Moreover,
since%
\begin{align}
\left[  Q_{i-1},Q_{i}\right]   &  =\beta_{i-1}\beta_{i}\left[  \nabla
_{i}P,\Delta_{i}P\right]  >0\label{SignalQ}\\
\left[  \nabla_{i}Q,\Delta_{i}Q\right]   &  =\left[  \left[  Q_{i-1}%
,Q_{i}\right]  P_{i},\left[  Q_{i},Q_{i+1}\right]  P_{i+1}\right]
=\beta_{i-1}\beta_{i}\beta_{i+1}\left[  \nabla_{i}P,\Delta_{i}P\right]
\left[  \Delta_{i}P,\Delta_{i+1}P\right]  >0\nonumber
\end{align}
we see that $Q$ is also convex and ordered in a counter-clockwise orientation
(in particular, all $\alpha_{i}$ and $\beta_{i}$ are positive). Finally, we
note for future reference that any scaling of a factor $\gamma$ applied to $P$
implies in a scaling of factor $\frac{1}{\gamma}$ applied to $Q$.%

\section{The curvature radius space\label{SecCurv}}

Consider all polygonal lines $M=\left(  M_{i}\right)  _{i\in\mathbb{Z}}$ whose
sides are respectively parallel to the corresponding sides of $P$ (though they
do not necessarily close, let us call them $P$\emph{-polygons} anyway). More
explicitly, we will require%
\[
\Delta_{i}M=r_{i}\cdot\Delta_{i}P
\]
for some list of real numbers $\left(  r_{i}\right)  _{i\in\mathbb{Z}}$ (we
allow $r_{i}=0$ with no further ado). Each number $r_{i}$ will be called the
\emph{curvature radius} of the side $M_{i}M_{i+1}$ (with respect to $P$). Up
to a translation, all $P$-polygons are uniquely represented by the list
$\left(  r_{i}\right)  $.

Before continuing, we add another restriction -- we require the sides (but not
necessarily the vertices!) to repeat:

\begin{definition}
A polygonal line $M$ is a \emph{periodic }$P$\emph{-polygon} when%
\[
\Delta_{i+2n}M=\Delta_{i}M
\]
for all $i\in\mathbb{Z}$. Since in this case we clearly have%
\begin{equation}
r_{i+2n}=r_{i} \label{Period}%
\end{equation}
each such polygonal line $M$ (up to a translation) can be represented by its
\emph{radii vector} $\vec{r}=\left(  r_{i}\right)  _{i=1,2,...,2n}$. We write
$L_{P}$ for such space of all periodic $P$-polygons.
\end{definition}

\begin{figure}[h] \centering
{\includegraphics[
trim=0.195823in 0.000000in 0.075256in 0.011954in,
height=1.2 in,width=2.4 in] {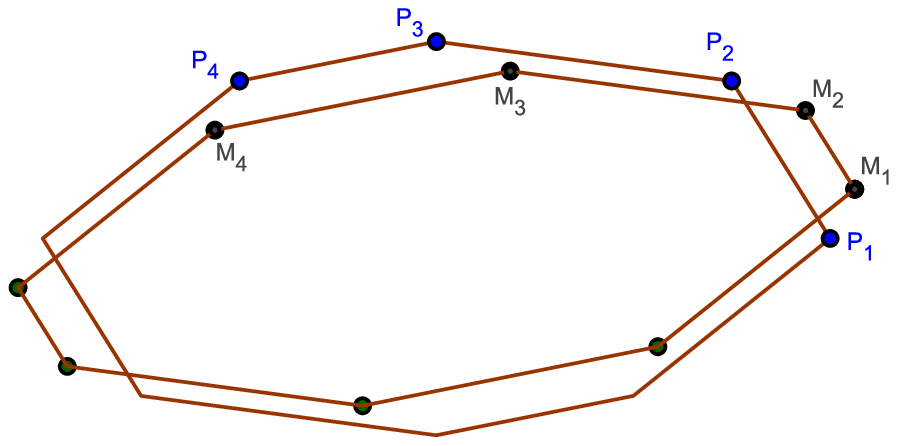}
\includegraphics[
trim=0.314505in 0.000000in 0.343497in 0.019234in,
height=1.2 in, width=2.4 in]{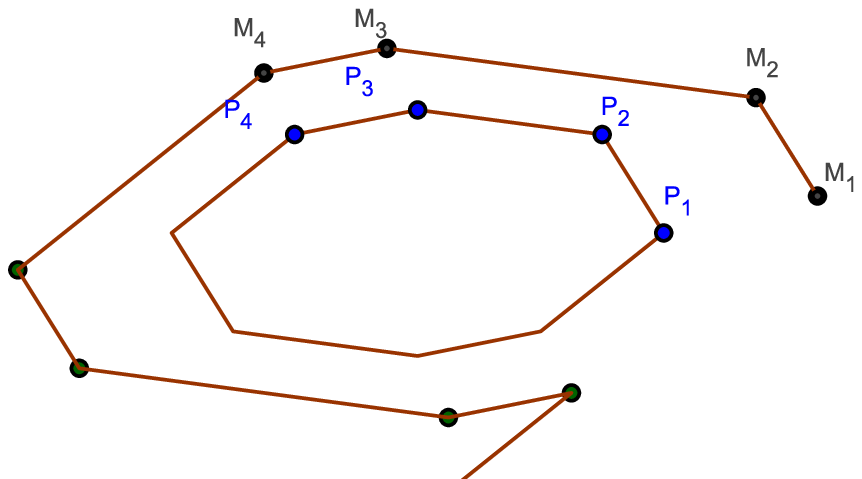}%
}
\caption{While $\vec{r}=\left( \frac{1}{2},1,\frac{3}{2},1,\frac{1}{2},1,\frac{3}{2},1\right)$ closes the line, $\vec{r}=\left(  1,2,1,2,1,2,1,-1\right)  $ does not}
\end{figure}

Our goal in this section is to pair up algebraic properties of $\vec{r}%
\in\mathbb{R}^{2n}$ with geometric properties of $M$ (compare this to the
similar analysis done in \cite{Evolutes}). We start defining a suitable inner
product in $L_{P}$:

\begin{definition}
Given two radii vectors $\vec{r}$ and $\vec{s}$ in $L_{P}$, we define their
$P$\emph{-inner product} as%
\[
\left\langle \vec{r},\vec{s}\right\rangle _{P}=\sum_{i=1}^{2n}\frac{r_{i}%
s_{i}}{\beta_{i}}%
\]

\end{definition}

Some interesting subspaces of $L_{P}$ are listed below:

\begin{itemize}
\item $C_{P}:$ the space of all \textbf{closed} polygons $M.$ Given our
periodicity condition on the sides, it is enough to check if $M_{2n+1}=M_{1}$,
that is%
\[
\sum_{i=1}^{2n}\Delta_{i}M=\vec{0}%
\]
or, in terms of the radii%
\begin{equation}
\sum_{i=1}^{2n}r_{i}\Delta_{i}P=\vec{0} \label{ClosedC}%
\end{equation}
Since this last condition is linear, $C_{P}$ is indeed a subspace of $L_{P}$,
and $\dim C_{P}=2n-2$ (since there must be two linearly independent
$\Delta_{i}P$).

\item $S_{P}:$ the space of all \textbf{symmetric} polygons $M$. Choosing the
origin as the center of symmetry, this condition translates to $M_{i+n}%
=-M_{i}$ for all $i$, or equivalently%
\[
r_{i+n}=r_{i}\text{\ \ \ \ \ }\left(  i=1,2,...,n\right)
\]
Clearly (both geometrically and algebraically), $S_{P}\subseteq C_{P}$, and
$\dim S_{P}=n$.

\item $A_{P}:$ the space of all \textbf{anti-symmetric} $P-$polygons $M$,
which we define algebraically by the condition%
\[
r_{i+n}=-r_{i}\ \ \ \ \ \left(  i=1,2,...,n\right)
\]
It is easy to see that $\dim A_{P}=n$; actually, under our inner product,%
\[
A_{P}=\left(  S_{P}\right)  ^{\perp}\text{.}%
\]

\item $D_{P}:$ the space of \textbf{double} polygons, that is, such that
$M_{i+n}=M_{i}$ for all $i$. This is equivalent to $M_{1}=M_{n+1}$ and
$\Delta_{i}M=\Delta_{i+n}M$ for all $i$, or%
\[
\left\{
\begin{array}
[c]{c}%
r_{i+n}=-r_{i}\ \ \ \ \ \left(  i=1,2,...,n\right) \\
\sum_{i=1}^{n}r_{i}\Delta_{i}P=0
\end{array}
\right.  \text{.}%
\]
Given the first condition, the second is equivalent to Equation \ref{ClosedC}
so $D_{P}=A_{P}\cap C_{P}$ and $\dim D_{P}=n-2$. In fact, since $D_{P}\perp
S_{P}$ and $\dim D_{P}+\dim S_{P}=\dim C_{P}$, we also have%
\[
C_{P}=S_{P}\oplus D_{P}%
\]

\end{itemize}

More specifically, $D_{P}$ is the orthogonal complement of $S_{P}$ in $C_{P}$.

\begin{itemize}
\item $B_{P}:$ the space of all \textbf{balls }(homothetic to $P$), which
consists of multiples of the vector $\mathbf{1}=\left(  1,1,...,1\right)  $.
Clearly $\dim B_{P}=1$ and $B_{P}\subseteq S_{P}$.
\end{itemize}

In order to further geometrically characterize subspaces of $L_{P}$, we turn
our attention to:

\begin{definition}
The \emph{support} associated to the side $M_{i}M_{i+1}$ is the signed
distance $h_{i}$ from the origin to that side, normalized to have value $1$
when the polygon is the $P$-ball itself. More explicitly:%
\[
h_{i}=\left[  M_{i},Q_{i}\right]  =\left[  M_{i+1},Q_{i}\right]
\]
The \emph{support function} is the list of values $\left\{  h_{i}\right\}
_{i\in\mathbb{Z}}$.
\end{definition}

\begin{proposition}
The radii vector depends linearly on the support function. Explicitly,%
\begin{equation}
r_{i}=h_{i}+\beta_{i}\nabla_{i}\left(  \alpha_{i}\Delta_{i}h\right)
\label{RandH}%
\end{equation}

\end{proposition}

\begin{proof}
Just use Eq. \ref{PQ} a few times:%
\begin{align*}
\alpha_{i}\Delta_{i}h  &  =\alpha_{i}\left[  M_{i+1},\Delta_{i}Q\right]
=\left[  P_{i+1},M_{i+1}\right]  \Rightarrow\\
&  \Rightarrow\beta_{i}\nabla_{i}\left(  \alpha_{i}\Delta_{i}h\right)
=\beta_{i}\left(  \left[  P_{i+1},\Delta_{i}M\right]  +\left[  \Delta
_{i}P,M_{i}\right]  \right)  =\\
&  =\beta_{i}\left[  P_{i+1},r_{i}\Delta_{i}P\right]  +\left[  Q_{i}%
,M_{i}\right]  =r_{i}-h_{i} {\tag*{$\square$}}
\end{align*}

\end{proof}

While the support function is not invariant by translations, the \emph{width
}of a polygon is:

\begin{definition}
The $P-$\emph{width} $w_{i}$ of $M$ between the sides $M_{i}M_{i+1}$ and
$M_{i+n}M_{i+n+1}$ is the signed distance between such sides, taking $P$ as
the unit reference ball. In other words%
\[
w_{i}=h_{i}+h_{i+n}=\left[  M_{i}-M_{i+n},Q_{i}\right]  =\left[
M_{i+1}-M_{i+n+1},Q_{i}\right]
\]

\end{definition}

Geometrically, $P$-polygons of constant $P$-width are characterized by:

\begin{proposition}
A $P$-polygon $M$ has constant $P$-width if and only if each of its
\textquotedblleft major\textquotedblright\ diagonals is parallel to the
corresponding major diagonal of $P$, that is,%
\[
\left[  M_{i}-M_{i+n},P_{i}\right]  =0\ \ \ \ \ \left(  i=1,2,...,n\right)
\text{.}%
\]
A $P$-polygon $M$ has constant $P$-width $0$ if and only if%
\[
M_{i+n}=M_{i}\ \ \ \ \ \left(  i=1,2,...,n\right)  \text{.}%
\]
In other words, $D_{P}$ is the space of $P$-polygons of width $0$.
\end{proposition}

\begin{proof}
For the first statement, we just need to remember once again Equation \ref{PQ}
and write%
\[
w_{i}=w_{i-1}\Leftrightarrow\left[  M_{i}-M_{i+n},Q_{i}\right]  =\left[
M_{i}-M_{i+n},Q_{i-1}\right]  \Leftrightarrow\left[  M_{i}-M_{i+n}%
,\Delta_{i-1}Q\right]  =0\Leftrightarrow\left[  M_{i}-M_{i+n},P_{i}\right]
=0
\]
For the second, just note that%
\begin{align*}
w_{i}  &  =0\Rightarrow\left[  M_{i}-M_{i+n},Q_{i}\right]  =0\\
w_{i}  &  =w_{i-1}\Rightarrow\left[  M_{i}-M_{i+n},P_{i}\right]  =0
\end{align*}
and, since $\left[  P_{i},Q_{i}\right]  =1$, these two equations are linearly
independent, implying $M_{i}=M_{i+n}$.
\qed \end{proof}

Adding support functions is the same as performing a \emph{Minkowski Sum} of
the corresponding polygons (see \cite{Thompson}, for example). Therefore, the
statement $C_{P}=S_{P}\oplus D_{P}$ can be translated as "every closed $P
$-polygon is the (Minkowski) sum of a symmetric polygon and a polygon of $0 $ width".

Though the support function cannot be determined by the radii, it is easy to
see that the width $\vec{w}$ depends linearly on $\vec{r}$, since%
\[
w_{i}=-\left[  M_{i+n}-M_{i+1},Q_{i}\right]  =-\left[  \sum_{k=i+1}%
^{i+n-1}\Delta_{k}M,Q_{i}\right]  =-\sum_{k=i+1}^{i+n-1}r_{k}\left[
\Delta_{k}P,Q_{i}\right]
\]
In particular, we have%
\[
\vec{w}\left(  \vec{r}+\lambda\mathbf{1}\right)  =\vec{w}\left(  \vec
{r}\right)  +\lambda\vec{w}\left(  \mathbf{1}\right)  =\vec{w}\left(  \vec
{r}\right)  +2\lambda\cdot\mathbf{1}%
\]
since the unit ball has constant width $2$. So, if $\vec{r}$ describes a
polygon with constant width $2\lambda_{0}$, then $\vec{r}-\lambda
_{0}\mathbf{1}$ will be a polygon with constant width $0$! This leads us to define:

\begin{itemize}
\item $W_{P}:$ the space of all \textbf{constant-width} $P$-polygons. Then
$W_{P}=D_{P}\oplus B_{P}$, and $\dim W_{P}=n-1.$ In particular, $W_{P}%
\subseteq C_{P}$.
\end{itemize}

\section{Evolutes and double evolutes\label{SecEvol}}

\subsection{Evolutes}

Given the definition of the curvature radii, we can \textquotedblleft
fit\textquotedblright\ a ball $P$ of radius $r_{i}$ to the side $\Delta_{i}M
$. We can then join the centers of such balls to form a new polygonal line:

\begin{definition}
The $P-$\emph{evolute} of $M$ is the polygonal line $E$ whose vertices are%
\begin{equation}
E_{i}=M_{i}-r_{i}P_{i}=M_{i+1}-r_{i}P_{i+1}\text{.} \label{Evolute}%
\end{equation}

\end{definition}

\begin{figure}[h]%
\centering
\includegraphics[height=1.5 in, width=3 in] {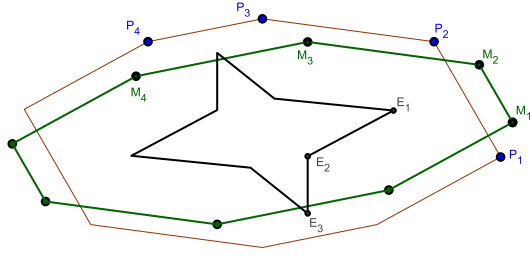}%
\caption{A polygon $M$ and its $P-$evolute $E$}%
\label{FigEvolute}%
\end{figure}

This definition is the discrete version of the evolute in \cite{Evolutes}.
Now, $E$ is a $Q$-polygon, since%
\[
\Delta_{i}E=\left(  M_{i+1}-r_{i+1}P_{i+1}\right)  -\left(  M_{i+1}%
-r_{i}P_{i+1}\right)  =-\Delta_{i}r\cdot P_{i+1}=\alpha_{i}\Delta_{i}%
r\cdot\Delta_{i}Q
\]
This means we can represent $E$ by its curvature radii \emph{with respect to
}$Q$, namely%
\begin{equation}
s_{i}=\alpha_{i}\Delta_{i}r \label{Evolute2}%
\end{equation}
So, using the radius representation, the evolute process $E_{P}:L_{P}%
\rightarrow L_{Q}$ is a linear transformation, whose $2n\times2n$ matrix can
be explicitly written as%
\[
E_{P}=\left[
\begin{array}
[c]{cccccc}%
-\alpha_{1} & \alpha_{1} & 0 & ... & 0 & 0\\
0 & -\alpha_{2} & \alpha_{2} & ... & 0 & 0\\
0 & 0 & -\alpha_{3} & ... & 0 & 0\\
\vdots & \vdots & \vdots & \ddots & \vdots & \vdots\\
0 & 0 & 0 & ... & -\alpha_{2n-1} & \alpha_{2n-1}\\
\alpha_{2n} & 0 & 0 & ... & 0 & -\alpha_{2n}%
\end{array}
\right]
\]
In order to characterize it geometrically, we need:

\begin{definition}
Given a periodic $P$-polygon $M$ represented by the radius vector $\vec{r}$,
its \emph{(signed) }$Q$\emph{-length }is%
\[
L_{Q}\left(  M\right)  =\left\langle r,\mathbf{1}\right\rangle _{P}=\sum
_{i=1}^{2n}\frac{r_{i}}{\beta_{i}}%
\]
which is the signed length of (one period of) $M$, taking $Q$ as the unit
ball, since (from Eq. \ref{PQ})%
\[
\Delta_{i}M=r_{i}\Delta_{i}P=\frac{r_{i}}{\beta_{i}}\cdot Q_{i}%
\]
Similarly, given a periodic $Q$-polygon $N$ represented by $\vec{s}$, its
\emph{(signed) }$P$\emph{-length} is%
\[
L_{P}\left(  N\right)  =\left\langle s,\mathbf{1}\right\rangle _{Q}=\sum
_{i=1}^{2n}\frac{s_{i}}{\alpha_{i}}%
\]
which is the signed length of $N$, taking $P$ as the unit ball, since%
\[
\Delta_{i}N=s_{i}\Delta_{i}Q=-\frac{s_{i}}{\alpha_{i}}\cdot P_{i+1}%
\]

\end{definition}

\begin{proposition}
The image of $E_{P}$ is the space of all $Q$-polygons of zero $P$-length, and
its kernel is $B_{P}$.
\end{proposition}

\begin{proof}
The kernel is easy:%
\[
s_{i}=0\Leftrightarrow\Delta_{i}r=0\Leftrightarrow r=\lambda\mathbf{1}%
\]
so we know that $rank\left(  E_{P}\right)  =2n-1$. Now, the $P$-length of the
evolute $E$ is%
\[
L_{P}\left(  E\right)  =\sum_{i=1}^{2n}\frac{s_{i}}{\alpha_{i}}=\sum
_{i=1}^{2n}\Delta_{i}r=0
\]
and since the condition $L_{P}\left(  N\right)  =0$ determines a subspace of
dimension $2n-1$ as well, we conclude it must be the whole image of $E_{P}$.
\qed \end{proof}

\subsection{Double evolutes}

But why stop there? If $E$ is a $Q$-polygon, we can find the evolute of $E$
taking $Q$ as the new reference ball! Namely, we find a new curve $F$ given
by:%
\[
F_{i+1}=E_{i+1}-s_{i}Q_{i+1}=E_{i}-s_{i}Q_{i}%
\]
where we shifted the indices in $F$ to compensate for the two forward
differences we have taken. We have%
\[
\Delta_{i}F=\left(  E_{i}-s_{i}Q_{i}\right)  -\left(  E_{i}-s_{i-1}%
Q_{i}\right)  =-\nabla_{i}s\cdot Q_{i}=-\beta_{i}\nabla_{i}s\cdot\Delta_{i}P
\]
so $F$ is a $P$-polygon again! The matrix of this second evolute transform
$E_{Q}:L_{Q}\rightarrow L_{P}$ is%
\[
E_{Q}=\left[
\begin{array}
[c]{cccccc}%
-\beta_{1} & 0 & 0 & ... & 0 & \beta_{1}\\
\beta_{2} & -\beta_{2} & 0 & ... & 0 & 0\\
0 & \beta_{3} & -\beta_{3} & ... & 0 & 0\\
\vdots & \vdots & \vdots & \ddots & \vdots & \vdots\\
0 & 0 & 0 & ... & -\beta_{2n-1} & 0\\
0 & 0 & 0 & ... & \beta_{2n} & -\beta_{2n}%
\end{array}
\right]
\]

\begin{figure}[h]%
\centering
\includegraphics[
height=1.5 in,
width=3 in
]%
{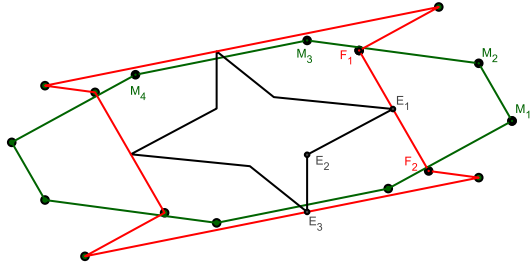}%
\caption{A polygon $M$ and its double evolute $F$}%
\label{FigDoubleEvo}%
\end{figure}

\begin{definition}
With the notation above, $F$ is the \emph{double evolute} of $M$. Explicitly,
$F$ is the curve whose curvature radii are:%
\[
t_{i}=-\beta_{i}\nabla_{i}s=-\beta_{i}\nabla_{i}\left(  \alpha_{i}\Delta
_{i}r\right)
\]
so the matrix of the \emph{double evolute transform} $T_{P}$ is%
\begin{equation}
T_{P}=E_{Q}E_{P}=\left[
\begin{array}
[c]{ccccc}%
\beta_{1}\left(  \alpha_{1}+\alpha_{2n}\right)  & -\alpha_{1}\beta_{1} & 0 &
... & -\alpha_{2n}\beta_{1}\\
-\alpha_{1}\beta_{2} & \beta_{2}\left(  \alpha_{2}+\alpha_{1}\right)  &
-\alpha_{2}\beta_{2} & ... & 0\\
0 & -\alpha_{2}\beta_{3} & \beta_{3}\left(  \alpha_{3}+\alpha_{2}\right)  &
... & 0\\
\vdots & \vdots & \vdots & \ddots & \vdots\\
0 & 0 & 0 & ... & -\alpha_{2n-1}\beta_{2n-1}\\
-\alpha_{2n}\beta_{2n} & 0 & 0 & ... & \beta_{2n}\left(  \alpha_{2n}%
+\alpha_{2n-1}\right)
\end{array}
\right]  \label{T}%
\end{equation}

\end{definition}

We quickly note that $T_{P}$ is invariant by ball rescaling, since a rescaling
on $P$ implies in the reverse scaling on $Q$. We are now ready to justify our
choices of inner products:

\begin{proposition}
For any vectors $\vec{r}\in L_{P}$ and $\vec{w}\in L_{Q}$, we have%
\[
\left\langle E_{P}\vec{r},\vec{w}\right\rangle _{Q}=\left\langle \vec{r}%
,E_{Q}\vec{w}\right\rangle _{P}%
\]
that is, $E_{Q}=E_{P}^{\ast}$ with this choice of inner products.
\end{proposition}

\begin{proof}
Just write%
\begin{align*}
\left\langle E_{P}\vec{r},\vec{w}\right\rangle _{Q}  &  =\sum_{i=1}^{2n}%
\frac{\alpha_{i}\Delta_{i}r}{\alpha_{i}}\cdot w_{i}=\sum_{i=1}^{2n}\Delta
_{i}r\cdot w_{i}\\
\left\langle \vec{r},E_{Q}\vec{w}\right\rangle _{P}  &  =\sum_{i=1}^{2n}%
\frac{r_{i}}{\beta_{i}}\cdot\left(  -\beta_{i}\nabla_{i}w\right)  =-\sum
_{i=1}^{2n}r_{i}\nabla_{i}w
\end{align*}
and the two sums are just rearrangements of each other.
\qed \end{proof}

\begin{proposition}
$T_{P}$ is self-adjoint (therefore non-negative), and $C_{P}$ is invariant by
$T_{P}$.
\end{proposition}

\begin{proof}
The first statement follows directly from $T_{P}=E_{P}^{\ast}E_{P}$. More
explicitly, we have%
\begin{equation}
\left\langle T_{P}r,r\right\rangle _{P}=\sum_{i=1}^{2n}r_{i}^{2}\left(
\beta_{i}+\beta_{i-1}\right)  -\sum_{i=1}^{2n}2\beta_{i}r_{i}r_{i+1}%
=\sum_{i=1}^{2n}\beta_{i}\left(  r_{i}-r_{i+1}\right)  ^{2}\geq0 \label{NN}%
\end{equation}
with equality if, and only if, $r$ is a multiple of $\mathbf{1}$ (remember
that $\beta_{i}>0$ from Eq. \ref{SignalQ}). The second fact is geometrically
clear, but we also offer an algebraic proof: just remember that $C_{P}$ is
defined by $\sum r_{i}\cdot\Delta_{i}P=0$. But the expression on the left side
remains invariant under the evolute transform, since%
\[
\sum r_{i}\cdot\Delta_{i}P=-\sum\Delta_{i}r\cdot P_{i+1}=-\sum\frac{s_{i}%
}{\alpha_{i}}P_{i+1}=\sum s_{i}\cdot\Delta_{i}Q\text{.} {\tag*{$\square$}}
\]

\end{proof}

\section{Discrete cycloids\label{SecCycl}}

\begin{definition}
A \emph{discrete cycloid} is a polygonal line $M$ which is homothetic to its
double evolute.
\end{definition}

In other words, we want%
\[
\vec{t}=T_{P}\vec{r}=\lambda\vec{r}\text{,}%
\]
or, in operator notation,%
\begin{equation}
\beta_{i}\nabla_{i}\left(  \alpha_{i}\Delta_{i}r\right)  +\lambda r_{i}=0
\label{EqOper}%
\end{equation}
for some constant $\lambda$ -- an eigenvalue problem!

\begin{example}
[Regular Polygons]\label{ExReg}If $P$ is a regular polygon with $2n$ sides,
then so is $Q$. Let $\gamma=\frac{\pi}{2n}$. We might as well rescale $P$ so
$P$ and $Q$ are congruent to each other -- explicitly, this happens when
\[
\left\vert P_{i}\right\vert =\left\vert Q_{i}\right\vert =\frac{1}{\sqrt
{\cos\gamma}}\Rightarrow\alpha_{i}=\beta_{i}=\frac{1}{2\sin\gamma}\text{.}%
\]
So in this case the double evolute transform%
\[
T_{P}=\frac{1}{4\sin^{2}\gamma}\left[
\begin{array}
[c]{cccccc}%
2 & -1 & 0 & ... & 0 & -1\\
-1 & 2 & -1 & ... & 0 & 0\\
0 & -1 & 2 & ... & 0 & 0\\
\vdots & \vdots & \vdots & \ddots & \vdots & \vdots\\
0 & 0 & 0 & ... & 2 & -1\\
-1 & 0 & 0 & ... & -1 & 2
\end{array}
\right]
\]
is a discrete convolution! One can verify directly that%
\begin{equation}
\vec{c}_{k}=\left(  1,\cos2k\gamma,\cos4k\gamma,...,\cos2\left(  2n-1\right)
k\gamma\right)  ;\ k=0,1,2,...,2n-1 \label{RegularEV}%
\end{equation}
are eigenvectors, since for $\phi\in\left\{  0,2k\gamma,4k\gamma,...\right\}
$:%
\[
2\cos\phi-\cos\left(  \phi+2k\gamma\right)  -\cos\left(  \phi-2k\gamma\right)
=2\left(  1-\cos2k\gamma\right)  \cos\phi=4\sin^{2}k\gamma\cos\phi
\]
In other words, the eigenvalues are%
\[
\lambda_{k}=\frac{\sin^{2}k\gamma}{\sin^{2}\gamma}%
\]
for $k=0,1,2,...,2n-1$. Since $\lambda_{k}=\lambda_{2n-k}$, all of them are
double eigenvalues, except for $\lambda_{0}=0$ and $\lambda_{n}=\csc^{2}%
\gamma$. Also, note that $\vec{c}_{k}\in S_{P}$ if $k$ is even, and $\vec
{c}_{k}\in A_{P}$ if $k$ is odd. In other words, if we rename the pair
$\left(  \lambda_{k},\lambda_{n-k}\right)  $ as $\left(  \lambda_{k}%
^{1},\lambda_{k}^{2}\right)  $ for $k=1,2,...,n-1$, the eigenvalues can be
ordered this way:%
\[
\lambda_{0}\left(  =0\right)  <\lambda_{1}^{1}=\lambda_{1}^{2}=1<\lambda
_{2}^{1}=\lambda_{2}^{2}<\lambda_{3}^{1}=\lambda_{3}^{2}<...\lambda_{n-1}%
^{1}=\lambda_{n-1}^{2}<\lambda_{n}\left(  =\csc^{2}\gamma\right)
\]

\end{example}

Our goal is now to discover which parts of the spectral structure above remain
true in the general case. Before we do that, let us rephrase our problem in a
slightly different way -- as a recurrence.

\subsection{The half-turn transform}

Once a candidate value for $\lambda$ is given, we can see our problem as a
recurrence on the coordinates of $\vec{r}$, namely%
\begin{equation}
r_{k+1}=\left(  1+\frac{\alpha_{k-1}}{\alpha_{k}}-\frac{\lambda}{\alpha
_{k}\beta_{k}}\right)  r_{k}-\frac{\alpha_{k-1}}{\alpha_{k}}r_{k-1}\text{
\ \ \ }\left(  k=2,3,...\right)  \label{Rec}%
\end{equation}
or, in matricial form%
\begin{equation}
\left[
\begin{array}
[c]{c}%
r_{k}\\
r_{k+1}%
\end{array}
\right]  =S_{k}\left(  \lambda\right)  \left[
\begin{array}
[c]{c}%
r_{k-1}\\
r_{k}%
\end{array}
\right]  \text{ where }S_{k}\left(  \lambda\right)  =\left[
\begin{array}
[c]{cc}%
0 & 1\\
-\frac{\alpha_{k-1}}{\alpha_{k}} & 1+\frac{\alpha_{k-1}}{\alpha_{k}}%
-\frac{\lambda}{\alpha_{k}\beta_{k}}%
\end{array}
\right]  \label{RecM}%
\end{equation}
Since we want a solution $\vec{r}\in L_{P}$, the trick is to find special
values of $\lambda$, $r_{1}$ and $r_{2}$ such that $r_{2n+1}=r_{1}$ and
$r_{2n+2}=r_{2}$ in this recurrence. Actually, in order to separate symmetric
from anti-symmetric solutions, we should stop half-way and check the
relationship between $\left(  r_{1},r_{2}\right)  $ and $\left(
r_{n+1},r_{n+2}\right)  $.

\begin{definition}
Given $\lambda\in\mathbb{R}$, the \emph{half-turn transform }$H_{\lambda}$ is
the linear transformation which takes $\vec{\rho}=\left(  r_{1},r_{2}\right)
\in\mathbb{R}^{2}$ to $H_{\lambda}\vec{\rho}=\left(  r_{n+1},r_{n+2}\right)
\in\mathbb{R}^{2}$ according to the recurrence (\ref{Rec}) above. In other
words%
\[
H_{\lambda}=S_{n+1}\left(  \lambda\right)  \cdot S_{n}\left(  \lambda\right)
\cdot...\cdot S_{2}\left(  \lambda\right)  \text{.}%
\]

\end{definition}

Though $H_{\lambda}$ depends on our index choice (for example, $H_{\lambda} $
does \textbf{not} necessarily take $\left(  r_{2},r_{3}\right)  $ to $\left(
r_{n+2},r_{n+3}\right)  $), we note that $H_{\lambda}$ \textbf{does} take
$\left(  r_{n+1},r_{n+2}\right)  $ to $\left(  r_{2n+1},r_{2n+2}\right)  $,
because of the $n$-periodicity of the sequences $\alpha$ and $\beta$. Now we
can summarize:

\begin{itemize}
\item "Finding an eigenvector $\vec{r}\in L_{P}$ of $T_{P}$ (for the
eigenvalue $\lambda$)" is equivalent to\newline"finding an eigenvector
$\vec{\rho}\in\mathbb{R}^{2}$ of $H_{\lambda}^{2}$ (for the eigenvalue $1$)";

\item "Finding an eigenvector $\vec{r}\in S_{P}$ of $T_{P}$ (for the
eigenvalue $\lambda$)" is equivalent to\newline"finding an eigenvector
$\vec{\rho}\in\mathbb{R}^{2}$ of $H_{\lambda}$ (for the eigenvalue $1$)";

\item "Finding an eigenvector $\vec{r}\in A_{P}$ of $T_{P}$ (for the
eigenvalue $\lambda$)" is equivalent to\newline"finding an eigenvector
$\vec{\rho}\in\mathbb{R}^{2}$ of $H_{\lambda}$ (for the eigenvalue $-1$)".
\end{itemize}

\begin{proposition}
\label{Cases}For any $\lambda\in\mathbb{R}$%
\[
\det H_{\lambda}=1
\]
This restricts $H_{\lambda}$ to six possibilities, according to the
\textbf{geometric} multiplicity of its eigenvalues:\newline Case $1$. If
$H_{\lambda}=I$, then $\lambda$ is a double eigenvalue of $T_{P}$, and both
its eigenvectors are in $S_{P}$.\newline Case $2$. If $H_{\lambda}=-I$, then
$\lambda$ is a double eigenvalue of $T_{P}$, and both its eigenvectors are in
$A_{P}$.\newline Case $3$. If $H_{\lambda}$ has a single eigenvalue $1$, then
$\lambda$ is a single eigenvalue of $T_{P}$, with its eigenvector in $S_{P}%
$.\newline Case $4$. If $H_{\lambda}$ has a single eigenvalue $-1$, then
$\lambda$ is a single eigenvalue of $T_{P}$, with its eigenvector in $A_{P}%
$.\newline Case $5$. If $H_{\lambda}$ has two distinct real eigenvalues $\mu$
and $\mu^{-1}$, then $\lambda$ is not an eigenvalue of $T_{P}$.\newline Case
$6$. If $H_{\lambda}$ has two complex eigenvalues $e^{\pm i\theta}$, then
$\lambda$ is not an eigenvalue of $T_{P}$.
\end{proposition}

\begin{proof}
The determinant follows directly from%
\[
\det S_{k}\left(  \lambda\right)  =\frac{\alpha_{k-1}}{\alpha_{k}}%
\]
and the symmetry of our ball ($\alpha_{n+1}=\alpha_{1}$). The existence of the
eigenvectors in cases $1$-$4$ follows from the observation before the
proposition -- we just have to make sure there are no \textbf{other}
eigenvectors of $H_{\lambda}^{2}$ in cases $3$-$6$.\newline In cases $3$ and
$4$, we have%
\[
H_{\lambda}\approx\left[
\begin{array}
[c]{cc}%
\pm1 & 1\\
0 & \pm1
\end{array}
\right]  \Rightarrow H_{\lambda}^{2}\approx\left[
\begin{array}
[c]{cc}%
1 & \pm2\\
0 & 1
\end{array}
\right]
\]
which clearly has only one eigenvalue $1$. In case $5$ the eigenvalues of
$H_{\lambda}^{2}$ are $\mu^{2}$ and $\mu^{-2}$, and neither is $1$. Finally,
in case $6$:%
\[
H_{\lambda}\approx R_{\theta}\Rightarrow H_{\lambda}^{2}\approx R_{2\theta}%
\]
where $R_{\theta}$ is a rotation of angle $\theta$ in the plane. Since
$\theta$ is not a multiple of $\pi$, $H_{\lambda}^{2}$ has no eigenvalues.
\qed \end{proof}

Now we go back to the structure of the eigenvalues of the double evolute
$T_{P}$.

\subsection{Eigenvalue $0$}

We already know from Eq. \ref{NN} that%
\[
\left\langle T_{P}r,r\right\rangle =\sum_{i=1}^{2n}\beta_{i}\left(
r_{i}-r_{i+1}\right)  ^{2}%
\]
so $T_{P}\vec{r}=0$ if and only if $\vec{r}=\gamma\mathbf{1.}$ In other words,
$Ker\left(  T_{P}\right)  =B_{P}$, and $0$ is always a single eigenvalue of
$T_{P}$.

\subsection{Double eigenvalue $1$}

Surprisingly, we can state very explicitly a pair of eigenvectors associated
to the $1$ eigenvalue of $T_{P}$ in the general case:

\begin{proposition}
Let $X$ be any non-zero fixed vector in $\mathbb{R}^{2}$. Then the vector $u$
defined by%
\[
u_{i}=\left[  X,Q_{i}\right]  \ \ \ \ \ \ \left(  i=1,2,...,2n\right)
\]
satisfies $T_{P}u=u$. Moreover, the $P$-polygon determined by $u$ does not
close, that is, $\vec{u}\notin C_{P}$.
\end{proposition}

\begin{proof}
Using our previous notation and remembering \ref{PQ}, we calculate directly%
\begin{align*}
s_{i}  &  =\alpha_{i}\Delta_{i}r=\alpha_{i}\left[  X,\Delta_{i}Q\right]
=-\left[  X,P_{i+1}\right] \\
t_{i}  &  =-\beta_{i}\nabla_{i}s=\beta_{i}\left[  X,\Delta_{i}P\right]
=\left[  X,Q_{i}\right]
\end{align*}
Now the "drift" of the $P$-polygon after one turn (see Eqs. \ref{ClosedC} and
\ref{PQ}) is%
\[
D=\sum_{i=1}^{2n}u_{i}\Delta_{i}P=\sum_{i=1}^{2n}\frac{\left[  X,Q_{i}\right]
}{\beta_{i}}Q_{i}\text{.}%
\]
So the component of the drift in the direction orthogonal to $X$ is%
\[
\left[  X,D\right]  =\sum_{i=1}^{2n}\frac{\left[  X,Q_{i}\right]  ^{2}}%
{\beta_{i}}%
\]
which is clearly positive, since $\beta_{i}>0$ for all $i$ and $X\neq\vec{0}$.
In other words, the "drift" cannot be $\vec{0}$, and the $P$-polygon does not close.
\qed \end{proof}

Since there are $2$ degrees of freedom in the choice of $X$, that gives us a
$2$ dimensional space of eigenvectors. More explicitly, one can define vectors
$u^{\prime}$ and $u^{\prime\prime}$ by taking%
\begin{align*}
u_{i}^{\prime}  &  =\left[  Q_{1},Q_{i}\right] \\
u_{i}^{\prime\prime}  &  =\left[  Q_{2},Q_{i}\right]
\end{align*}
which are clearly linearly independent: $u^{\prime}=\left(  0,\left[
Q_{1},Q_{2}\right]  ,...\right)  $ and $u^{\prime\prime}=\left(  \left[
Q_{2},Q_{1}\right]  ,0,...\right)  $. Also, from the symmetry of $Q$, we see
that such solutions satisfy $u_{i+n}=-u_{i}$, so we have%
\[
Ker\left(  T_{P}-I\right)  \subseteq A_{P}%
\]

Actually, consider the following general discrete Sturm-Liouville equation
presented in \cite{WangShi}:%
\begin{equation}
\nabla_{i}\left(  a_{i}\Delta_{i}u\right)  +\left(  \lambda b_{i}%
-c_{i}\right)  u_{i}=0 \label{DSL}%
\end{equation}
where $a_{i},b_{i}>0$ and $c_{i}$ are given $n$-periodic sequences and we want
to find the eigenvalue $\lambda$ and the $2n$-periodic sequence $u$. A class
of such equations can be interpreted as cycloid problems:

\begin{proposition}
Suppose the problem described above has $c_{i}=0$ and a double eigenvalue $1$.
Then there is a symmetric, locally convex ball $P$ such that%
\[
a_{i}=\alpha_{i}=\frac{1}{\left[  Q_{i},Q_{i+1}\right]  }\text{ and }%
b_{i}=\frac{1}{\beta_{i}}=\left[  P_{i},P_{i+1}\right]
\]
so Eq. \ref{DSL} becomes the cycloid problem in Eq. \ref{EqOper}.
\end{proposition}

\begin{proof}
Let $u=x$ and $u=y$ be two linearly independent solutions of \ref{DSL}
associated to the eigenvalue $1$. Take the sequence of points $Q_{i}=\left(
x_{i},y_{i}\right)  $ in the plane. Looking at one coordinate at a time, it is
easy to see that%
\[
\nabla_{i}\left(  a_{i}\Delta_{i}Q\right)  =-b_{i}Q_{i}\Rightarrow\nabla
_{i}a\cdot\Delta_{i}Q-a_{i-1}\cdot\delta_{i}^{2}Q=-b_{i}Q_{i}%
\]
so, remembering that%
\[
\left[  \delta_{i}^{2}Q,\Delta_{i}Q\right]  =\left[  \Delta_{i}Q-\nabla
_{i}Q,\Delta_{i}Q\right]  =\left[  \Delta_{i}Q,\nabla_{i}Q\right]
\]
we can write%
\begin{equation}
a_{i-1}\left[  \delta_{i}^{2}Q,\Delta_{i}Q\right]  =b_{i}\left[  Q_{i}%
,\Delta_{i}Q\right]  \Rightarrow a_{i-1}\left[  \Delta_{i}Q,\nabla
_{i}Q\right]  =b_{i}\left[  Q_{i},Q_{i+1}\right]  \label{Coeff}%
\end{equation}
Note that, if we had $\left[  \Delta_{i}Q,\nabla_{i}Q\right]  =0$ above, that
would force $\left[  Q_{i},Q_{i+1}\right]  =0$, which is not possible since
$x$ and $y$ are linearly independent. Now, aiming towards Eq. \ref{PQ}, we
define%
\[
P_{i}=-\frac{\nabla_{i}Q}{\left[  Q_{i-1},Q_{i}\right]  }%
\]
so we can compute%
\[
\left[  \nabla_{i}Q,\Delta_{i}Q\right]  =\left[  P_{i},P_{i+1}\right]
\cdot\left[  Q_{i-1},Q_{i}\right]  \cdot\left[  Q_{i},Q_{i+1}\right]
\]
That allows us to rewrite Eq. \ref{Coeff} as%
\[
a_{i-1}\left[  Q_{i-1},Q_{i}\right]  =\frac{b_{i}}{\left[  P_{i}%
,P_{i+1}\right]  }\text{,}%
\]
so, rescaling $Q$ if necessary, we may assume $a_{i-1}=\alpha_{i-1}=\frac
{1}{\left[  Q_{i-1},Q_{i}\right]  }$ and $b_{i}=\frac{1}{\beta_{i}}=\left[
P_{i},P_{i+1}\right]  $, as claimed. Since the sequences $a_{i}$ and $b_{i}$
are positive and $n$-periodic, we can now see that $P$ and $Q$ are locally
convex and symmetric.
\qed \end{proof}

\subsection{No triple eigenvalues}

\begin{proposition}
If $\lambda$ is any eigenvalue, then%
\[
\dim\left(  Ker\left(  T_{P}-\lambda I\right)  \right)  \leq2
\]

\end{proposition}

\begin{proof}
This is a direct consequence of the recurrence -- given $\lambda$, once the
values $r_{1}$ and $r_{2}$ are chosen, recurrence \ref{Rec} determines all
other coordinates of $\vec{r}$. Or, in other words, $H_{\lambda}^{2}$ has at
most $2$ eigenvectors associated to the eigenvalue $1$.
\qed \end{proof}

\subsection{No eigenvalues with eigenvectors in both $S_{P}$ and $A_{P}$}

\begin{proposition}
The restrictions $T_{P}|_{A_{P}}$ and $T_{P}|_{S_{P}}$ have no common eigenvalues.
\end{proposition}

\begin{proof}
If a eigenvalue $\lambda$ had eigenvectors in both $A_{P}$ and $S_{P}$, the
corresponding half-way transform $H_{\lambda}$ would have both $1$ and $-1$ as
eigenvalues. Since $\det H_{\lambda}=1$, this is impossible.
\qed \end{proof}

Putting all pieces together, we have the main result of this section:

\subsection{The spectral structure of the double evolute}

\begin{proposition}
The eigenvalues of $T_{P}$ (in Eq. \ref{T}) can be ordered as%
\begin{equation}
\lambda_{0}\left(  =0\right)  <\lambda_{1}^{1}=\lambda_{1}^{2}\left(
=1\right)  <\lambda_{2}^{1}\leq\lambda_{2}^{2}<\lambda_{3}^{1}\leq\lambda
_{3}^{2}<...\lambda_{n-1}^{1}\leq\lambda_{n-1}^{2}<\lambda_{n} \label{SpecL}%
\end{equation}
where $\lambda_{k}^{1,2}$ are eigenvalues of $T_{P}|_{A_{P}}$ if $k$ is odd
and $0$ and $\lambda_{k}^{1,2}$ are eigenvalues of $T_{P}|_{S_{P}}$ if $k$ is even.
\end{proposition}

\begin{proof}
Any $P$-ball can be continuously deformed towards a $2n$-regular polygon
(being kept convex and symmetric in the process). Throughout the process, the
eigenvalues of $T_{P}$ change continuously, and its eigenvectors (always in
$A_{P}$ or $S_{P}$) can also be chosen to change continuously. Now, the
proposition above guarantees that eigenvalues corresponding to eigenvectors in
distinct spaces (one in $A_{P}$, another in $S_{P}$) cannot "switch places"
through the process! Therefore, the ordering of the eigenvalues of $T_{P}$
(belonging to different spaces) must be the same as it is in the case of
regular polygons!
\qed \end{proof}

\begin{remark}
The above proposition begs the following question: can one hear the shape of a
convex symmetric body\footnote{Ok, there is no physical \emph{hearing} in this
context, but we wanted to cite \cite{Kac1966}.}? We mean, given a specific
list of cycloid eigenvalues, are we able to determine the shape of the
$P$-ball? Or, a slight variation on this question: if all eigenvalues are
double, do we necessarily have a regular polygon?
\end{remark}

\begin{remark}
Note that Eq. \ref{RandH} can be rewritten in terms of the double evolute
transform!%
\[
\vec{r}=\left(  I-T_{P}\right)  \vec{h}%
\]
So, let $\vec{r}$ be the radii associated to a \textbf{closed} cycloid, say,
$T_{P}\vec{r}=\lambda\vec{r}$ with $\lambda\neq1$. Taking $\vec{h}=\frac
{1}{1-\lambda}\vec{r}$, we define a $P$-polygon whose radii vector is exactly
$\vec{r}$, since%
\[
\left(  I-T_{P}\right)  \vec{h}=\frac{1}{1-\lambda}\vec{r}-\frac{\lambda
}{1-\lambda}\vec{r}=\vec{r}\text{.}%
\]
This shows that the support function of a (correctly placed in the plane)
closed cycloid is also an eigenvector of $T_{P}$. In other words, any closed
$P$-polygon can be written as the Minkowski sum of $2n-2$ closed cycloids!
Now, periodic support functions cannot represent open polygons like our open
cycloids, hence our choice or primarily working with curvature radii.
\end{remark}

\subsection{Cusps}

\begin{definition}
Given a periodic $P$-polygon $M$ represented by the radius vector $\vec{r}$,
the \emph{orientation} of its side $\Delta_{i}M$ is the sign of the
corresponding radius $r_{i}$. A vertex $V$ of $M$ is a \emph{cusp} if its
neighbor (non-degenerate) sides have opposite orientations. Such cusp will be
named \emph{ordinary} if there is at most one degenerate side at $V$.
\end{definition}

Such cusps have appeared at \cite{Schneider} (where they were called
\emph{strong corners)}. Geometrically, sides which meet at cusp are on
opposite sides of the normal line $\left\{  M_{i}+tP_{i};t\in\mathbb{R}%
\right\}  $; algebraically, each cusp corresponds to a zero-\textbf{crossing}
of the sequence $\left(  r_{i}\right)  _{i\in\mathbb{Z}}$. For example, a
snippet $\left(  ...,-2,0,0,0,0,3,...\right)  $ corresponds to \textbf{one}
non-ordinary cusp, while $\left(  ...,2,0,2,...\right)  $ is not a cusp at all.

Now, consider how the number of zero-crossings (per period) of a sequence
$\vec{r}$ can change if $\vec{r}$ is changed continuously. One can create two
cusps going from a "$+0+$" subsequence to "$+-+$" (or from "$-0-$" to "$-+-$",
of course); one can destroy two cusps reversing this process. Finally, many
zero-crossings can be created at once if a sequence of consecutive $0$s is
present (for example, from "$+000+$" to "$+-+-+$"). Outside of these
situations, there is no way to create or destroy a cusp (it is possible to
\emph{move it}, of course, going from "$++-$" to "$+0-$" to "$+--$", but in
each of these cases we have only one ordinary cusp). Under this light, the
following proposition is important:

\begin{proposition}
In a cycloid, any zero entries in the radius vector must correspond to cusps;
also, all cusps are ordinary.
\end{proposition}

\begin{proof}
This is a direct consequence of the recurrence \ref{Rec}: if $v_{k}=0$, then
$v_{k+1}$ and $v_{k-1}$ must have opposite signs; if two consecutive sides
were $0$, all of them would be $0$.
\qed \end{proof}

\begin{proposition}
Given a $P$-ball with $2n$ sides, the number of cusps of its associated
cycloids is respectively%
\[
0,2,2,4,4,6,6,...,2n-2,2n-2,2n
\]

\end{proposition}

\begin{proof}
Once again, deform $P$ continuously towards a $2n-$regular polygon. Since
ordinary cusps are stable with relation to changes in the radii vectors (as
long as no consecutive zeroes occur), the previous proposition guarantees that
each cycloid will have a constant number of cusps as the deformation takes
place. Now it is just a matter of checking how many cusps each of the
eigenvectors in Equation \ref{RegularEV} has.
\qed \end{proof}

\section{Other periods\label{SecPeri}}

In the Euclidean plane, hypocycloids and epicycloids might take several turns
to close. This suggests we could relax the periodicity condition (Eq.
\ref{Period}) on the curvature representation $\left(  r_{i}\right)
_{i\in\mathbb{Z}}$ of the polygonal line $M$ -- requiring, instead, the
sequence to be periodic with period $2mn$, say. Can we find discrete closed
cycloids with other periods this way?

We claim that a big part of the analysis in such cases is already done! After
all, we did not really use that the unit ball $P$ is a \textbf{simple} closed
convex polygon -- we only needed local convexity, as seen in Eqs.
\ref{SignalQ}, so we could establish the positivity of $\alpha_{i}$ and
$\beta_{i}$ (defined in Eq. \ref{AlphaandBeta}). In other words, if one wants
to find $2mn$-periodic cycloids with reference to a $2n$-ball $P$, one can
instead look for $2mn$-periodic cycloids with reference to the $2mn$-ball
which is determined by $P$ \textbf{traversed }$\mathbf{m}$\textbf{\ times}
(call this polygon $mP$).

So the reader will have the pleasure of re-reading this article from the
beginning switching $P$ with $mP$, a polygon which goes $m$ times around the
origin (two articles for the price of one!). All calculations in Section
\ref{SecMink} are unchanged, except that indices go $i=1,2,...,2mn$. The new
curvature radius space of Section \ref{SecCurv} (that would be $L_{mP}$) has
dimension $2mn$, and the corresponding spaces $S_{mP}$ and $A_{mP}$ are still
orthogonal complements of each other. All calculations done in Section
\ref{SecEvol} still hold, but the matrices are $2mn\times2mn$. Finally, all
arguments in Section \ref{SecCycl} still hold, with a few exceptions -- first,
our base case must change\footnote{In fact, all the theory could be done if
$P$ were any locally convex polygon that goes around the origin $m$ times (not
necessarily repeating itself at each turn), but then the geometric
intepretation of $P$ as a unit ball is somewhat diminished. \textbf{Three}
articles for the price of one!}:

\begin{example}
[Regular Polygon traversed $m$ times]If $P$ is a regular polygon with $2n$
sides, traversed $m$ times, then so is $Q$. Write $\alpha=\frac{\pi}{2n}$ and
$\gamma=\frac{\pi}{2mn}$ and assume by rescaling that%
\[
\left\vert P_{i}\right\vert =\left\vert Q_{i}\right\vert =\frac{1}{\sqrt
{\cos\alpha}}\Rightarrow\alpha_{i}=\beta_{i}=\frac{1}{2\sin\alpha}\text{.}%
\]
The double evolute transform is the same as before, except for the matrix size
which now must be $2mn\times2mn$:%
\[
T_{mP}=\frac{1}{4\sin^{2}\alpha}\left[
\begin{array}
[c]{cccccc}%
2 & -1 & 0 & ... & 0 & -1\\
-1 & 2 & -1 & ... & 0 & 0\\
0 & -1 & 2 & ... & 0 & 0\\
\vdots & \vdots & \vdots & \ddots & \vdots & \vdots\\
0 & 0 & 0 & ... & 2 & -1\\
-1 & 0 & 0 & ... & -1 & 2
\end{array}
\right]
\]
The eigenvectors are%
\[
\vec{d}_{k}=\left(  1,\cos2k\gamma,\cos4k\gamma,...,\cos2\left(  2mn-1\right)
k\gamma\right)  ;\ k=0,1,2,...,2mn-1
\]
and the eigenvalues are%
\[
\sigma_{k}=\frac{\sin^{2}k\gamma}{\sin^{2}\alpha}%
\]
for $k=0,1,2,...,2mn-1$. Again, $\sigma_{k}=\sigma_{2mn-k}$, so all of them
are double eigenvalues, except for $\sigma_{0}=0$ and $\sigma_{mn}=\csc
^{2}\alpha$. Renaming $\left(  \sigma_{k},\sigma_{2mn-k}\right)  $ as $\left(
\sigma_{k}^{1},\sigma_{k}^{2}\right)  $ for $k=1,2,...,mn-1$, the eigenvalues
can be ordered this way:%
\[
\sigma_{0}\left(  =0\right)  <\sigma_{1}^{1}=\sigma_{1}^{2}<\sigma_{2}%
^{1}=\sigma_{2}^{2}<...<\sigma_{m}^{1}=\sigma_{m}^{2}(=1)<...<\sigma
_{mn-1}^{1}=\sigma_{mn-1}^{2}<\sigma_{mn}\left(  =\csc^{2}\alpha\right)
\]
Each eigenvector has only regular cusps. In fact, taking $k=0,1,2,...,mn$, we
can see that the number of cusps in $\vec{d}_{k}$ is exactly $2k$. So the
number of cusps in each cycloid can be ordered (correspondingly to the
eigenvalues) in the list%
\[
0,2,2,4,4,...,2mn-2,2mn-2,2mn\text{.}%
\]
Finally, if $k=jm$ ($j=1,2,...,n$) then $k\gamma=j\alpha$, and%
\[
\vec{d}_{jm}=\vec{c}_{j}\text{ and }\sigma_{jm}^{1,2}=\lambda_{j}^{1,2}%
\]
where $\lambda$ and $\vec{c}$ are the eigenvalues/vectors in the case the
polygon was traversed just once (see Example in page \pageref{ExReg}).
\end{example}

So the following (partial!) result is easily obtained as before:

\begin{proposition}
When the unit ball is traversed $m$ times, the eigenvalues of the double
evolute transform can be ordered as%
\[
\sigma_{0}\left(  =0\right)  <\sigma_{1}^{1}\leq\sigma_{1}^{2}<\sigma_{2}%
^{1}\leq\sigma_{2}^{2}<...<\sigma_{m}^{1}=\sigma_{m}^{2}=1<...<\sigma
_{mn-1}^{1}\leq\sigma_{mn-1}^{2}<\sigma_{mn}%
\]
where $\sigma_{k}^{1,2}$ are eigenvalues of $T_{mP}|_{A_{mP}}$ if $k$ is odd
and $0$ and $\sigma_{k}^{1,2}$ are eigenvalues of $T_{mP}|_{S_{mP}}$ if $k$ is
even. The number of cusps of the associated cycloids are respectively%
\[
0,2,2,4,4,...,2mn-2,2mn-2,2mn\text{.}%
\]

\end{proposition}

\begin{proof}
As before, start with a $2n-$regular polygon, traversed $m$ times around the
origin, and deform it con\-ti\-nuous\-ly towards $mP$. Since the eigenvalues
$\sigma_{k}^{1,2}$ cannot switch places with $\sigma_{k+1}^{1,2}$ (that would
imply an eigenvalue common to $A_{mP}$ and $S_{mP}$), the ordering above must
be kept throughout. Similarly, since all cusps are kept ordinary throughout
the deformation, their number must be constant in each eigenvector.
\qed \end{proof}

\subsection{From one turn to many turns}

Our final goal this section is to relate the eigenvalues of $T_{mP}$ with the
eigenvalues of $T_{P}$. To do that, we fully turn our attention to the
half-turn transform, for now we have:%
\[
\left[
\begin{array}
[c]{c}%
r_{2mn+2}\\
r_{2mn+1}%
\end{array}
\right]  =H_{\lambda}^{2m}\left[
\begin{array}
[c]{c}%
r_{1}\\
r_{2}%
\end{array}
\right]
\]
So now we say:

\begin{itemize}
\item "Finding an eigenvector $\vec{r}=\left(  r_{1},r_{2},...,r_{2mn-1}%
\right)  \in L_{mP}$ of $T_{mP}$ (for the eigenvalue $\lambda$)" is equivalent
to\newline"finding an eigenvector $\vec{\rho}=\left(  r_{1},r_{2}\right)
\in\mathbb{R}^{2}$ of $H_{\lambda}^{2m}$ (for the eigenvalue $1$)";
\end{itemize}

Adapting proposition \ref{Cases}, we have:

\begin{proposition}
Let $\lambda\in\mathbb{R}$. We can once again classify $\lambda$ as an
eigenvalue of $T_{mP}$ according to the \textbf{geometric} multiplicity of the
eigenvalues of $H_{\lambda}$:\newline Case $1$. If $H_{\lambda}=I$, then
$\lambda$ is a double eigenvalue of $T_{mP}$, and both its eigenvectors are in
$S_{mP}$.\newline Case $2$. If $H_{\lambda}=-I$, then $\lambda$ is a double
eigenvalue of $T_{mP}$, and both its eigenvectors are in $A_{mP}$.\newline
Case $3$. If $H_{\lambda}$ has a single eigenvalue $1$, then $\lambda$ is a
single eigenvalue of $T_{mP}$, with its eigenvectors in $S_{mP}$.\newline Case
$4$. If $H_{\lambda}$ has a single eigenvalue $-1$, then $\lambda$ is a single
eigenvalue of $T_{mP}$, with its eigenvectors in $A_{mP}$.\newline Case $5$.
If $H_{\lambda}$ has two distinct real eigenvalues $\mu$ and $\mu^{-1}$, then
$\lambda$ is not an eigenvalue of $T_{mP}$.\newline Case $6a$. If $H_{\lambda
}$ has two complex eigenvalues $e^{\pm i\theta}$ and $\theta$ is an even(odd)
multiple of $\frac{\pi}{m}$, then $\lambda$ is a \textbf{double} eigenvalue of
$T_{mP}$, with its eigenvectors in $S_{mP} $($A_{mP}$).\newline Case $6b$. If
$H_{\lambda}$ has two complex eigenvalues $e^{\pm i\theta}$ and $\theta$ is
not a multiple of $\frac{\pi}{m}$, then $\lambda$ is not an eigenvalue of
$T_{mP}$.
\end{proposition}

\begin{proof}
Cases $1$ and $2$ are as before. Cases $3$ and $4$ are also very much the
same, since%
\[
H_{\lambda}\approx\left[
\begin{array}
[c]{cc}%
\pm1 & 1\\
0 & \pm1
\end{array}
\right]  \Rightarrow H_{\lambda}^{2m}\approx\left[
\begin{array}
[c]{cc}%
1 & \pm2m\\
0 & 1
\end{array}
\right]
\]
which clearly has only one eigenvalue $1$. In case $5$ the eigenvalues of
$H_{\lambda}^{2m}$ are $\mu^{2m}$ and $\mu^{-2m}$, and neither is $1$. Now, in
case $6$:%
\[
H_{\lambda}\approx R_{\theta}\Rightarrow H_{\lambda}^{2m}\approx R_{2m\theta}%
\]
and that is why we have to separate it further: if $2m\theta$ is a multiple of
$\pi$, then $H_{\lambda}^{2m}=\pm I$ and we are back to cases $1$ or $2$;
otherwise, $H_{\lambda}^{2m}$ has no eigenvalues.
\qed \end{proof}

Do note that the half-turn transform $H_{\lambda}$ is exactly the same here as
it was in the "one turn" case! So cases $1$-$5$ happen just as often here as
they did before, and with the same values for $\lambda$ -- they account for
$2m$ of the $2mn$ eigenvalues we found! So all the new eigenvalues must come
from case $6a$... Can we figure out their ordering with relation to the "old"
eigenvalues? Indeed we can -- we just need another continuity argument.

\begin{proposition}
Let the specter of $T_{P}$ be as denoted in \ref{SpecL}. The eigenvalue(s)
$\mu$ (and $\mu^{-1}$) of $H_{\lambda}$ depend on $\lambda$ in the following
way:\newline a) If $\lambda_{k}^{2}<\lambda<\lambda_{k+1}^{1}$, then
$H_{\lambda}$ has two complex eigenvalues (as in case $6$). In fact, as
$\lambda$ grows from $\lambda_{k}^{2}$ to $\lambda_{k+1}^{1}$, the eigenvalues
$\mu$ and $\mu^{-1}$ go through all values in the complex unit circle exactly
once.\newline b) If $\lambda_{k}^{1}<\lambda<\lambda_{k}^{2}$ or $\lambda<0$,
then $H_{\lambda}$ has two distinct real eigenvalues (as in case $5$).
\end{proposition}

\begin{proof}
Just consider the positioning of $\mu$ in the complex plane in each of the $6$
cases we considered, as displayed in Figure \ref{FigComplex}:

\begin{center}
\begin{figure}[h] \centering
\includegraphics[
height=1.5 in,
width=3 in
]%
{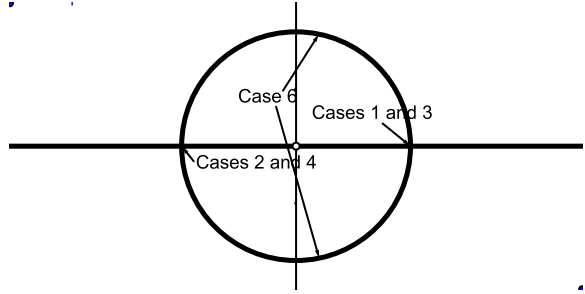}%
\caption{Where is $\mu$ in the complex plane?}
\label{FigComplex}%
\end{figure}
\end{center}

Since $\mu$ varies continuously with $\lambda$ (and $\mu\neq0$), in order to
go from cases $1,3$ to $2,4$ (and vice-versa) we must go through all complex
numbers in the circle $\left\vert \mu\right\vert =1$ (they do come in
conjugate pairs, of course). So each interval of the kind $I_{k}=\left(
\lambda_{k}^{2},\lambda_{k+1}^{1}\right)  $ will contain at least $m-1$ values
of $\lambda$ which correspond to the eigenvalues $e^{\pm i\theta}$ for
$\theta=\frac{k\pi}{m}$ ($k=1,2,...,m-1$) of $H_{\lambda}$, as in case $6a$.
Now, how do we know that each of these complex eigenvalues is visited only
once as $\lambda$ varies in $I_{k}$? Every time $\lambda$ falls into case
$6a$, such $\lambda$ is an eigenvalue of $T_{mP}$. Since there are $n$ such
intervals $I_{k}$, and each interval already contains $m-1$ \textbf{double}
eigenvalues, this already accounts for $2\left(  m-1\right)  n$ eigenvalues.
If we now add the $\lambda_{k}^{1,2}$ themselves (there are $2n$ of them),
which are also eigenvalues of $T_{mP}$, we are already crowded with all $2mn$
eigenvalues $T_{mP}$ could possibly have! So no other values of $\lambda$ can
generate eigenvalues of $H_{\lambda}$ of the form $e^{\pm i\theta}$ where
$\theta$ is \textbf{any} rational multiple of $\pi$. That proves not only that
each complex eigenvalue is visited only once in each interval (a), but also
shows that in (b) no new complex eigenvalues $\mu$ can appear -- so while
$\lambda\in\left(  \lambda_{k}^{1},\lambda_{k}^{2}\right)  $ or $\lambda<0$,
we must keep $\mu$ real and different from $1$.
\qed \end{proof}

We can now gather all information we have in one final proposition:

\begin{proposition}
The $2mn$ eigenvalues of $T_{mP}$ can be ordered the following way%
\[%
\setlength{\tabcolsep}{0pt}
\begin{tabular}
{llllll}
&  &  &  & $\lambda_{0}\left(  =0\right)  $ & $<$\\
$<\lambda_{1/m}^{1}=\lambda_{1/m}^{2}$ & $<\lambda_{2/m}^{1}=\lambda_{2/m}%
^{2}$ & $<...$ & $<\lambda_{1-1/m}^{1}=\lambda_{1-1/m}^{2}<$ & $\lambda
_{1}^{1}=\lambda_{1}^{2}\left(  =1\right)  $ & $<$\\
$<\lambda_{1+1/m}^{1}=\lambda_{1+1/m}^{2}$ & $<\lambda_{1+2/m}^{1}%
=\lambda_{1+2/m}^{2}$ & $<...$ & $<\lambda_{2-1/m}^{1}=\lambda_{2-1/m}^{2}<$ &
$\lambda_{2}^{1}\leq\lambda_{2}^{2}$ & $<$\\
$<\lambda_{2+1/m}^{1}=\lambda_{2+1/m}^{2}$ & $<\lambda_{2+2/m}^{1}%
=\lambda_{2+2/m}^{2}$ & $<...$ & $<\lambda_{3-1/m}^{1}=\lambda_{3-1/m}^{2}<$ &
$\lambda_{3}^{1}\leq\lambda_{3}^{2}$ & $<$\\
$<...$ & $...$ & $...$ & $...$ & $...$ & $<$\\
$<\lambda_{n-1+1/m}^{1}=\lambda_{n-1+1/m}^{2}$ & $<\lambda_{n-1+2/m}%
^{1}=\lambda_{n-1+2/m}^{2}$ & $<...$ & $<\lambda_{n-1/m}^{1}=\lambda
_{n-1/m}^{2}<$ & $\lambda_{n}$ &
\end{tabular}
\]

\end{proposition}

\begin{proof}
All the work is already done -- just define $\lambda_{p}^{1,2}$ as the value
of $\lambda$ which makes $H_{\lambda}$ have the eigenvalues $e^{\pm i\pi p}$
when the integer part of $p$ is even (that is, when you are moving from cases
$1,3$ to $2,4$); or $e^{\pm i\pi\left(  1-p\right)  }$ otherwise.
\qed \end{proof}

\subsection{No periods!}

What if we relax the periodicity condition even further: let us \textbf{not}
require the list $\left\{  r_{n}\right\}  _{n\in Z}$ to be periodic. What then?

First of all, clearly the eigenvalues can now be any real number. After all,
just pick any $\lambda\in\mathbb{R}$, any two values $r_{1}$, $r_{2}%
\in\mathbb{R}$ and apply recurrence \ref{Rec} both backwards and forwards to
create the complete list $\left\{  r_{n}\right\}  $. Moreover, if the number
$\lambda$ you picked is any of the eigenvalues of $T_{mP}$ for \emph{some} $m
$, the list will be periodic of period $2mn$, as seen above (and, unless
$\lambda=1$, the cycloid will close). Otherwise, we must have non-periodic
cycloids! As such, one interesting phenomenon (which does not exist in the
Euclidean case) can occur -- a \emph{spiraling cycloid}.

\begin{center}
\begin{figure}[h] \centering
{\includegraphics[
trim=0.176516in 0.000000in 0.176910in 0.000980in,
height=1.2 in,
width=2.4 in
]%
{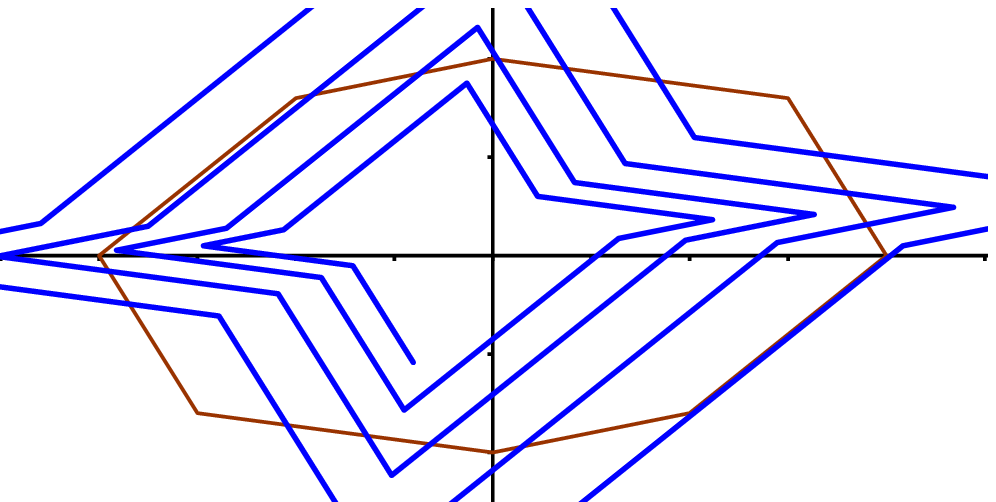}%
\includegraphics[
trim=0.176122in 0.000000in 0.176516in 0.000000in,
height=1.2 in,
width=2.4 in
]%
{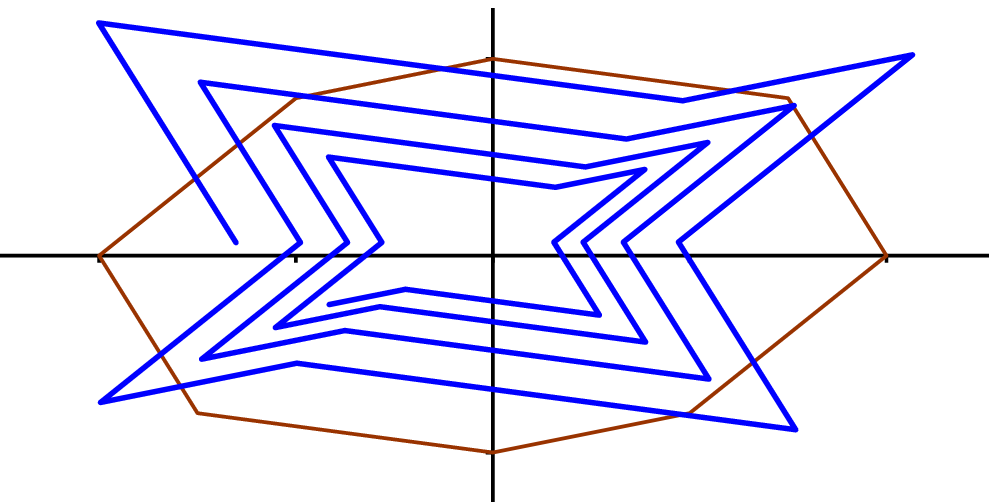}%
}
\caption{Two very different spiraling cycloids
associated to the same $\lambda\approx3.3$}
\end{figure}
\end{center}

\begin{proposition}
Suppose $\lambda_{n}^{1}<\lambda<\lambda_{n}^{2}$ or $\lambda<0$. Then the
cycloids associated to $\lambda$ are unlimited.
\end{proposition}

\begin{proof}
Just remember that in this case, the eigenvalues associated to $H_{\lambda}$
can be written as $\mu$ and $\mu^{-1}$ where $\left\vert \mu\right\vert >1 $.
Writing $\vec{\rho}=\left(  r_{1},r_{2}\right)  =c_{1}\rho_{1}+c_{2}\rho_{2}$
where $\rho_{1}$ and $\rho_{2}$ are the respective eigenvectors of
$H_{\lambda}$ we have $H_{\lambda}^{m}\vec{\rho}=c_{1}\mu^{m}\rho_{1}+c_{2}%
\mu^{-m}\rho_{2}$. If $c_{1}\neq0$, we have $\left\vert r_{mn+1}\right\vert
\rightarrow\infty$ as $m\rightarrow\infty$; if $c_{1}=0$, then $\left\vert
r_{mn+1}\right\vert \rightarrow\infty$ as $m\rightarrow-\infty$. Either way,
the cycloid is unlimited.
\qed \end{proof}

\section{A Four Vertex Theorem\label{SecFourVertex}}

In differential geometry, a \emph{"vertex"} of a curve is a point where the
curvature reaches a local extremum. We want to show that any closed
$P$-polygon has at least $4$ vertices... of this other kind, which needs to be defined.

\begin{definition}
An \emph{edgex} of a $P$-polygon is a collection of adjacent sides whose
(equal) curvatures correspond to a strict extremum of the sequence of radii in
$\vec{r}$. In other words, an edgex is a zero-crossing of $\Delta\vec{r}$.
\end{definition}

To clarify, if the radii vector is $\left(  ...,1,2,3,3,3,2,2,...\right)  $,
we count that sequence of threes as \textbf{one} edgex, but if the sequence
were $\left(  ....,1,2,3,3,3,4,5,...\right)  $ we see no edgex at all in this
part of the polygon. So we are ready to state our "four edgex theorem", which
adapts the reasoning in \cite{Taba}:

\begin{proposition}
Any closed convex polygon $M$ (which is not homothetic to the $P$-ball) has at
least four edgices. If $M$ has constant $P$-width, it must have at least six edgices.
\end{proposition}

\begin{proof}
Take the radii vector $\vec{r}\in C_{P}$ associated to $M$, and decompose it
as%
\[
\vec{r}=\vec{r}_{0}+\vec{r}_{2}^{1}+\vec{r}_{2}^{2}+\vec{r}_{3}^{1}+\vec
{r}_{3}^{2}+...+\vec{r}_{n-1}^{1}+\vec{r}_{n-1}^{2}+\vec{r}_{n}%
\]
where each $\vec{r}_{i}^{j}$ is a cycloid associated to the eigenvalue
$\lambda_{i}^{j}$ (note that $i\neq1;$ suppose for now that $\vec{r}_{2}%
^{1}\neq\vec{0}$). Since $\vec{r}_{0}$ is a multiple of $\mathbf{1}$, it does
not alter its number of local extrema, so we may discard it completely. We can
define a \emph{double} \emph{involute of }$M$ as the $P$-polygon with radii
vector given by%
\[
I_{P}\vec{r}=\frac{\vec{r}_{2}^{1}}{\lambda_{2}^{1}}+\frac{\vec{r}_{2}^{2}%
}{\lambda_{2}^{2}}+...+\frac{\vec{r}_{n}}{\lambda_{n}}%
\]
Note that\emph{\ }$T_{P}I_{P}\vec{r}=\vec{r}$ ($I_{P}$ is the pseudo-inverse
of $T_{P}$). Now the key to the proof is to realize that every application of
either evolute transform $E_{P}$ or $E_{Q}$ cannot decrease the number of
edgices! This should be clear from Eq. \ref{Evolute2} -- since $s_{i}%
=\alpha_{i}\Delta_{i}r$, each edgex of $\vec{r}$ must correspond to a zero
crossing of $\vec{s}$, and between two consecutive zero crossings of $\vec{s}
$ we must have an edgex of $\vec{s}$! In other words, since $\vec{r}$ is the
double evolute of $I_{P}\vec{r}$, we conclude that $\vec{r}$ must have at
least as many edgices as $I_{P}\vec{r}$. Now iterate $I_{P}$! So $\vec{r}$ has
at least as many edgices as%
\[
I_{P}^{k}\vec{r}=\frac{\vec{r}_{2}^{1}}{\left(  \lambda_{2}^{1}\right)  ^{k}%
}+\frac{\vec{r}_{2}^{2}}{\left(  \lambda_{2}^{2}\right)  ^{k}}+...+\frac
{\vec{r}_{n}}{\left(  \lambda_{n}\right)  ^{k}}=\frac{1}{\left(  \lambda
_{2}^{1}\right)  ^{k}}\left(  \vec{r}_{2}^{1}+\left(  \frac{\lambda_{2}^{1}%
}{\lambda_{2}^{2}}\right)  ^{k}\vec{r}_{2}^{2}+...+\left(  \frac{\lambda
_{2}^{1}}{\lambda_{n}}\right)  ^{k}\vec{r}_{n}\right)
\]
We might as well ignore the homothety of a factor of $\left(  \lambda_{2}%
^{1}\right)  ^{k}$, and note that eventually this involute will be arbitrarily
close to $\vec{r}_{2}^{1}$ -- which is a cycloid with $4$ cusps, and therefore
$4$ edgices (if $\lambda_{2}^{1}=\lambda_{2}^{2}$, just group together
$\vec{r}_{2}^{1}$ and $\vec{r}_{2}^{2}$ to form a single cycloid with $4$
cusps and repeat the argument). If it so happens that $\vec{r}_{2}^{1}=\vec
{r}_{2}^{2}=\vec{0}$, just repeat the argument using the first non-zero
cycloid instead of $\vec{r}_{2}^{1}$, and the number of cusps will be even
bigger. For example, if the initial curve has constant $P$-width, then it must
live in $W_{P}=B_{P}\oplus D_{P}$; since we are ignoring the component $\vec{r}_{0}$,
we have a vector in $D_{P}\subseteq A_{P}$, so the $\vec{r}_{2}$
components \textbf{must} be zero and the decomposition starts with $\vec
{r}_{j}^{i}$ where $j\geq3$ -- a cycloid with $6$ cusps or more, and therefore
$6$ edgices or more.
\qed \end{proof}

\end{document}